\documentclass[11pt,reqno,a4paper]{amsart}
\usepackage{amssymb, amscd, amsmath, enumerate,verbatim, upref}

\usepackage{color}
\usepackage{graphicx}

\usepackage{html}

\input xy
\xyoption{all}

\numberwithin{equation}{subsection}
\newtheorem{theorem}{Theorem}[subsection]
\newtheorem{lemma}[theorem]{Lemma}
\newtheorem{proposition}[theorem]{Proposition}
\newtheorem{corollary}[theorem]{Corollary}

\theoremstyle{definition}

\newtheorem{definition}[theorem]{Definition}
\newtheorem{remark}[theorem]{Remark}

\newtheorem{example}[theorem]{Example}
\newtheorem{examples}[theorem]{Examples}

\newcommand{\A}{\mathcal{A}}
\newcommand{\B}{\mathcal{B}}
\newcommand{\C}{\mathcal{C}}
\newcommand{\D}{\mathcal{D}}

\newcommand{\F}{\mathcal{F}}
\newcommand{\G}{\mathcal{G}}
\newcommand{\Hcal}{\mathcal{H}}

\newcommand{\Ocal}{\mathcal{O}}

\newcommand{\Scal}{\mathcal{S}}

\newcommand{\W}{\mathcal{W}}
\newcommand{\X}{\mathcal{X}}
\newcommand{\Y}{\mathcal{Y}}
\newcommand{\Z}{\mathcal{Z}}

\newcommand{\Ab}{\mathbf{Ab}}

\newcommand{\Set}{\mathbf{Set}}
\newcommand{\Simpl}{\mathbf{\Delta}}
\newcommand{\Simplop}{\mathbf{\Delta}^\op}
\newcommand{\Sset}{s\mathbf{S}}

\newcommand{\Sp}{\mathbf{Sp}}

\newcommand{\loc}[2]{\mathcal{#1}[\mathcal{#2}^{-1}]}

\newcommand{\rdf}{\mathbb{R}}

\newcommand{\quis}{\textit{quis}}

\newcommand{\Cochains}[1]{\mathbf{C}^*(#1)}
\newcommand{\Cochainsp}[1]{\mathbf{C}^{\geq b}(#1)}

\newcommand{\dCochains}[1]{\mathbf{C}^{**}(#1)}

\newcommand{\Sheaf}{\mathbf{Sh}}

\newcommand{\Sh}[2]{\mathbf{Sh}(#1,#2)}
\newcommand{\PrSh}[2]{\mathbf{PrSh}(#1,#2)}

\newcommand{\y}{\mathbf{y}}

\newcommand{\Tot}{\mathbf{Tot}}

\newcommand{\dirlim}{\underset{\longrightarrow}{\mathrm{colim}}\,}
\newcommand{\invlim}{\underset{\longleftarrow}{\mathrm{lim}}\,}
\newcommand{\holim}{\underset{\longleftarrow}{\mathrm{holim}}\,}

\newcommand{\op}{\mathrm{op}}

\newcommand{\id}{\mathrm{id}}

\newcommand{\simple}{\mathbf{s}}

\newcommand{\mc}[1]{\mathcal{#1}}
\newcommand{\mrm}[1]{\mathrm{#1}}
\newcommand{\mbf}[1]{\mathbf{#1}}
\newcommand{\mbb}[1]{\mathbb{#1}}
\newcommand{\Dl}{\ensuremath{\Simpl} }

\newcommand{\ds}{\displaystyle}

\begin{document}
\begin{large}

\title[Godement resolutions]{Godement resolutions and sheaf homotopy theory}

\author{Beatriz Rodr\'{\i}guez Gonz\'{a}lez}
\address{ICMAT\\ CSIC-Complutense-UAM-CarlosIII\\ Campus Cantoblanco, UAM. 28049 Madrid, Spain.}

\author{Agust\'{\i} Roig}
\address{Dept. Matem\`{a}tica Aplicada I\\ Universitat Polit\`{e}cnica de Catalunya, UPC \\ Diagonal 647, 08028 Barce\-lo\-na, Spain.}

\thanks{First named author partially supported by ERC Starting Grant project TGASS and by
contracts SGR-119 and FQM-218. Second named author partially supported by projects MTM2009-09557, 2009 SGR 119 and MTM2012-38122-C03-01/FEDER. 
To appear in Collectanea Mathematica. The final publication is available at Springer via 
\href{http://dx.doi.org/10.1007/s13348-014-0123-x}{http://dx.doi.org/10.1007/s13348-014-0123-x}}

\begin{abstract} The Godement cosimplicial resolution is available for a wide range of categories of sheaves. In this paper we investigate under which conditions of the Grothendieck site and the category of coefficients it can be used to obtain fibrant models and hence to do sheaf homotopy theory. For instance, for which Grothendieck sites and coefficients we can define sheaf cohomology and derived functors through it.
\end{abstract}

\date{\today}
\maketitle
\tableofcontents

\section{Introduction}

\subsubsection{} Godement resolutions have been an essential tool in sheaf homotopy theory and its applications almost from the start \cite{Go} and keep cropping up in different contexts: see for instance \cite{SGA4} for abelian sheaves on a Grothendieck site, \cite{Th} for sheaves of spectra on a Grothendieck site, \cite{N} for sheaves of (filtered) dg commutative algebras over topological spaces, \cite{MV} for simplicial sheaves on a Grothendieck site, \cite{SdS} for sheaves of $\Ocal_X$-modules over schemes, or \cite{GL}  and \cite{Ba}  for sheaves of DG-categories over schemes..., to name but a few.

In particular, the great flexibility of the \emph{cosimplicial} Godement resolution, together with its excellent functorial properties, appear to account for its omnipresence: in fact, in order to define it for a sheaf $\F : \X^\op \longrightarrow \D$ on a Grothendieck site with enough points $\X$ and values in some category of coefficients $\D$, we only need $\D$ to have filtered colimits and arbitrary products. In this situation, we obtain a functor
$$
G^\bullet :  \Sheaf (\X , \D) \longrightarrow \Delta \Sheaf (\X , \D)
$$
from sheaves on $\X$ with values in $\D$ to cosimplicial ones.

The question we address in this paper is the following: under which conditions for the Grothen\-dieck site $\X$ and the category of coefficients $\D$ can the cosimplicial Godement construction be used to transfer homotopical structure from $\D$ to the category of sheaves $\Sheaf (\X, \D)$?

\subsubsection{} Let us elaborate a little further.
Making use of a (realization of the) homotopy limit $\simple : \Delta\D \longrightarrow \D$, which we call a \textit{simple} functor,
we can \lq\lq reassemble" all the cosimplicial pieces of
$G^p\F$ obtaining a single sheaf which might be entitled to be a \lq\lq model" for $\F$. To get anchorage for her ideas,
the reader may think of $\D$ as being the category of cochain complexes of abelian groups $\Cochains{\Ab}$ and $\simple$ the total
complex of a double complex. In this way we obtain a sheaf together a universal map
\begin{equation}\label{hypercohomologysheaf}
\rho_\F : \F \longrightarrow \mathbb{H}_\X (\F) = \simple G^\bullet (\F) \ ,
\end{equation}
called here the \emph{hypercohomology sheaf of} $\F$ following
Thomason and Mitchell (\cite{Th}, \cite{Mit}).

So a particular instance of our initial question is the following: assume that $\X$ has a final object $X$, when would it make sense to define sheaf cohomology of $\X$ with coefficients in $\F$ as $\Gamma (X, \mathbb{H}_\X (\F))$? More precisely, we are asking when this formula would define a right derived functor in the sense of Quillen \cite{Q}; that is, a left Kan extension.

\subsubsection{} In order to talk about derived functors and homotopy categories, we need to specify the class of morphisms with respect to which we localize. In all the examples we are aware of, this is the class that keeps track of the topology of $\X$, the one of \emph{local equivalences}: we have a distinguished class of morphisms $\mathrm{E}$, or \lq\lq equivalences", in the category of coefficients $\D$; and, for a morphism of sheaves $\varphi : \F \longrightarrow \G$ to be called a local equivalence, we require every morphism induced on stalks $\varphi_x : \F_x \longrightarrow \G_x$ to be in $\mathrm{E}$. Let us note this class of local equivalences as $\W$. For instance, for $\D = \Cochains{\Ab}$ we could take $\mathrm{E}$ to be the class of \emph{quasi-isomorphisms}, $\quis$, morphisms which induce isomorphisms in cohomology.

\subsubsection{}  A first approach could be to study our question in the context of Quillen model categories. That is, to assume that our
coefficient category $(\D,\mrm{E})$ supports a Quillen model structure and that it induces one on $(\Sheaf (\X , \D),\W)$ in such a way that the Godement resolution becomes a fibrant model for every sheaf. As proved in
\cite{Be1} and \cite{Be2}, this is indeed possible under certain (non-trivial) hypotheses on the model category $(\D,\mrm{E})$.

Instead, we opt here to keep to the minimum the amount of structure
on $(\Sheaf (\X, \D),\W)$ necessary to have the sheaves $\mathbb{H}_\X (\F)$ as fibrant models.
This allows us to 1) cover a more general class of coefficient categories (e.g. filtered complexes
over any $($AB4$)^{\ast}$ and $($AB5$)$ abelian category) and 2) have more flexibility in the transference of
the resulting technique to the multiplicative setting. This last point is the subject of a forthcoming sequel to this paper,
where we transfer the results obtained here for $\Sheaf (\X, \D)$ to sheaf of operads and algebras
(over any operad) on $\D$, and their corresponding filtered versions.

One such minimal amount of structure is attained with \emph{Cartan-Eilenberg categories}, or CE-categories, for short: an approach to
homotopical algebra started in \cite{GNPR1} and further developed in \cite{P}, \cite{C1} and \cite{C2}. A (right) CE-category
consists of a category $\C$ endowed with two classes of distinguished morphisms, \emph{strong} and \emph{weak} equivalences, $\Scal \subset \W$, and a \emph{CE-fibrant} model for each object (see \ref{defCEcategory} for the precise definition). The name of these structures comes from the classic book \cite{CE}, where, in modern parlance, the homotopy theory of the category of cochain complexes $\Cochains{\Ab}$ is developed around two classes of distinguished morphisms: homotopy equivalences ($\Scal$) and $\quis$  ($\W$).

But CE-structures allow more freedom of choice for classes $\Scal$ and $\W$ than classical \lq\lq homotopy equivalences" and \lq\lq weak equivalences". This is particularly interesting for our categories of sheaves, for which the natural choices are:

\begin{itemize}
\item \emph{global} equivalences,  as $\Scal$: those morphisms of sheaves such that $\varphi (U) : \F (U) \longrightarrow \G (U)$ belongs to the class of equivalences $\mathrm{E}$ in $\D$ for every object (open set) $U \in \X$, and
\item \emph{local} equivalences, as $\W$: already mentioned above.
\end{itemize}

In order to provide our categories of sheaves $\Sheaf (\X , \D)$ with a CE-structure, we need very few elements in our category of
coefficients $\D$: essentially, our needs reduce to a class of equivalences $\mathrm{E}$ and a simple functor
$\simple : \Delta \D \longrightarrow \D$ which is a realization of the homotopy limit.
This is summarized in the notion of
\emph{descent category} (see \cite{Rod1}, \cite{Rod2} and the second section in this paper; cf. also \cite{GN}).

\subsubsection{} Our main result (Theorem \ref{CartanEilenbergSheaves}) provides equivalent conditions
guaranteeing that our initial question has a positive answer:

\begin{theorem} Let $\X$ be a Grothendieck site and $(\D, \mathrm{E})$ a descent
category satisfying the hypotheses \emph{(\ref{hipotesis})}. Then, the
following statements are equivalent:
\begin{enumerate}
 \item[\emph{(1)}] $(\Sh{\X}{\D}, \Scal, \mc{W})$ is a right Cartan-Eilenberg category and for every sheaf $\F \in \Sh{\X}{\D}$,
$\rho_\F : \F \longrightarrow \mbb{H}_{\X} (\F)$ is a CE-fibrant model.
 \item[\emph{(2)}] For every sheaf $\F\in \Sh{\X}{\D}$, $\rho_\F : \F\longrightarrow \mbb{H}_{\X} (\F)$ is in $\mc{W}$.
 \item[\emph{(3)}] The simple functor commutes weakly with stalks.
 \item[\emph{(4)}] For every sheaf $\F\in \Sh{\X}{\D}$, $\mbb{H}_{\X} (\F)$ satisfies Thomason's descent; that is, $\rho_{\mbb{H}_{\X} (\F)} :
\mbb{H}_{\X} (\F) \longrightarrow \mathbb{H}^2_\X (\F)$ is in $\Scal$.
\end{enumerate}
\end{theorem}

This theorem shows, first, that the existence of a CE-structure on the category of sheaves $\Sheaf (\X , \D)$ boils down to the property that for every sheaf $\F$ the universal arrow $\rho_\F : \F \longrightarrow \mathbb{H}_{\X} (\F)$ is a local equivalence (condition $(2)$). Hence, we need nothing else that this CE-structure to answer our problem; i.e., the fact that the Godement construction can be used to transfer homotopical structure from $\D$ to the category of sheaves $\Sh{\X}{\D}$ is equivalent to the existence of this CE-structure.

The theorem also shows that the fact of the Godement resolution being a CE-fibrant model is equivalent to Thomason's classic descent (for sheaves of spectra \cite{Th}, condition $(4)$; see also Corollary \ref{Thomasondescentproperty}). So being CE-fibrant is quite a natural and central notion for sheaves.

Finally, the theorem gives a down-to-earth equivalent condition for all this to happen, which will be the one we will use in practice: condition $(3)$ says that the simple functor $\simple$ must commute with stalks up to local equivalence. For instance, for bounded cochain complexes this is a consequence of the commutation of the total complex functor $\Tot $ with filtered colimits.

\subsubsection{Acknowledgements} This paper develops an idea suggested to us by Vicente Navarro. We owe him a debt of gratitude for sharing it with us. The second named
author also benefited from many fruitful conversations with Pere Pascual. We are indebted to Francisco Guill\'{e}n, Fernando Muro, Luis Narv\'{a}ez and Abd\'{o} Roig for their comments. People at sci.math.research and Mathoverflow made useful suggestions kindly answering our questions there.

\section{Homotopical preliminaries}

We introduce here the definitions and results concerning descent and CE-categories necessary for our paper. The interested reader may consult \cite{Rod1}, \cite{Rod2} and \cite{GNPR1} for further details.

\subsection{Descent categories}

\subsubsection{Notations} By $\Simpl$ we mean the \textit{simplicial category}. We denote by $\Simpl\mc{D}$ (resp. $\Simplop\mc{D}$) the category of cosimplicial (resp.  simplicial) objects in a fixed category $\mc{D}$. The \textit{diagonal functor} $\mrm{D}:\Simpl\Simpl\mc{D}\longrightarrow\Simpl\mc{D}$ is given by
$\mrm{D}(\{Z_{n,m}\}_{n,m\geq 0})=\{Z_{n,n}\}_{n\geq 0}$. The \textit{constant simplicial object} defined by $A\in\mc{D}$ will be denoted by $c(A)$ or by $A\times\Dl$.

\subsubsection{} A (cosimplicial) descent category consists, roughly, of a category $\D$ endowed with a class $\mrm{E}$ of
`weak equivalences' and with a `simple' functor $\simple : \Simpl\D \longrightarrow \D$ subject to the axioms below.
These axioms ensure that $\simple$ is a realization of the homotopy limit for cosimplicial objects, and that the localized
category $\D[\mrm{E}^{-1}]$ possesses a rich homotopical structure.

\begin{definition}\label{defdescentcategory}\cite[1.1]{Rod1} A \textit{$($cosimplicial$)$ descent category} is the data
$(\mc{D},\mrm{E},\mbf{s},\mu,\lambda)$ where $\mc{D}$ is a category closed under finite products and $\mrm{E}$ is a saturated class
of morphisms of $\mc{D}$, closed under finite products, called \textit{weak equivalences}. The triple $(s,\mu,\lambda)$ is subject to the following axioms:

\begin{enumerate}
\item[(S1)] The simple functor $\mbf{s}:\Dl\mc{D}\longrightarrow \mc{D}$ commutes with finite products up to equivalence.
That is, the canonical morphism $\mbf{s}(X\times Y)\longrightarrow \mbf{s}(X) \times \mbf{s}(Y)$ is in $\mrm{E}$ for all $X$, $Y$ in $\Simpl\mc{D}$.

\item[(S2)] $\mu:\mbf{s}\mathbf{s}\dashrightarrow\mbf{s}\mrm{D}$ is a zigzag of natural weak equivalences. Recall that $\mbf{s}\mrm{D}Z$ denotes the simple
of the diagonal of $Z$, while $\mbf{s}\mbf{s}Z=\mbf{s}(n\longrightarrow\mbf{s}(m\rightarrow
Z_{n,m}))$.

\item[(S3)] $\lambda:\id_{\mc{D}}\dashrightarrow\mbf{s}(-\times\Dl)$ is a zigzag of natural weak equivalences, which is assumed to be compatible with $\mu$ in the sense of \emph{op.cit.}.

\item[(S4)] If $f:X\longrightarrow Y$ is a morphism in $\Dl\mc{D}$ with $f_n \in \mrm{E}$ for all $n$,
then $\mbf{s}(f)\in\mrm{E}$.

\item[(S5)] The image under the simple functor of the cosimplicial map $A^{d_0}:A^{\Dl[1]}\longrightarrow A $ is a weak equivalence for each object $A$ of $\mc{D}$.
\end{enumerate}
\end{definition}

For the sake of brevity, we will also denote a descent category by
$(\mc{D},\mrm{E})$.

\begin{remark}\label{LambdaMuFlechas} The presence of zigzags in the definition of descent category is needed to ensure its homotopy invariance
(see \cite{Rod1}, Proposition 1.8). However, every example used in this paper has both $\mu$ and $\lambda$ as actual natural transformations (see Examples \ref{example1bounded} - \ref{example1simplicial}). Since this significantly simplifies exposition, we will assume they are so for the descent categories considered
throughout the paper. We will also assume that simple functors preserve limits. But this is not a major restriction: it is fulfilled by all our examples of descent categories so far.
\end{remark}

\subsubsection{} Among the hereditary results of descent categories, let us point out one we will be using time and again and whose proof we leave as an easy exercise for the interested reader:

\begin{lemma}[Transfer Lemma]\label{transferLema} Let $(\mc{D'},\mrm{E'},\mbf{s'},\mu',\lambda')$ be a descent category. Given a functor
$\psi:\mc{D}\longrightarrow\mc{D'}$, consider in $\mc{D}$ the weak equivalences $\mrm{E}=\psi^{-1}\mrm{E}'$. Assume that $\mc{D}$ has finite products and is equipped with a functor $\mbf{s}:\Simpl\mc{D}\longrightarrow \mc{D}$, together with
compatible natural weak equivalences $\mu:\mbf{s}\mbf{s}\longrightarrow\mbf{s}\mrm{D}$ and $\lambda:\id_{\mc{D}}\longrightarrow \mathbf{s}(-\times \Dl)$.
Then, $(\mc{D},\mrm{E},\mbf{s},\mu,\lambda)$ is a descent category provided the following statements hold:%
\begin{enumerate}
\item[$($FD1$)$] $\psi$ commutes with finite products up to equivalence. That is, the natural map $\psi(X\times Y) \longrightarrow \psi(X)\times \psi(Y)$ is in $\mrm{E}$ for all $X,Y$ in $\mc{D}$.
\item[$($FD2$)$] There exists a natural weak equivalence $\theta: \psi\,\mbf{s} \longrightarrow \mbf{s'}\,\psi $ filling the square
$$
\xymatrix@H=4pt@C=20pt@R=20pt{\Simpl\mc{D} \ar[r]^{\psi}\ar[d]_{\mbf{s}} & \Simpl\mc{D'} \ar[d]^{\mbf{s'}}\\
                                \mc{D}\ar[r]_{\psi}  \ar@{}[ru]|{\stackrel{\theta}{\mbox{\rotatebox[origin=c]{33}{$\Rightarrow$}}}}     {}  & {} \mc{D'}
                                }
$$
\end{enumerate}
\end{lemma}

\subsubsection{}\label{examples1} To end with, we describe some examples of descent categories.

\begin{example}\label{example1bounded} \textbf{Bounded complexes {$[$Rod1, (3.4)$]$}.} Let $\A$ be an abelian category.
For a fixed integer $b \in \mathbb{Z}$, denote by $\Cochainsp{\A}$ the category of uniformly bounded below cochain complexes of $\A$;
that is, $A^n = 0$ for all $n < b$ and all $A^* \in \Cochainsp{\A}$.

We will consider the following descent structure on $\Cochainsp{\A}$. The weak equivalences $\mrm{E}$ are the {\it quasi-isomorphism\/} (\quis ): those maps inducing isomorphism in cohomology.
The simple functor $\mbf{s}:\Simpl\Cochainsp{\A}\longrightarrow \Cochainsp{\A}$ at a given cosimplicial cochain complex
$A$ is the (product) total complex of the double complex induced by $A$:
$$
\mbf{s}(A)^n = \prod_{p+q =n} A^{pq} \ .
$$
Which, in this case, since $A$ has finite codiagonals, $\mbf{s}(A)^n =\bigoplus_{p+q =n} A^{pq}$. $\mu_Z$ is just the Alexander-Whitney map $\mbf{s}\mbf{s}Z\longrightarrow \mbf{s}\mrm{D}Z$ and $\lambda_X^n : X^n\longrightarrow \mbf{s}(X\times\Dl)^n$ is the canonical inclusion.
\end{example}

\begin{example}\label{example1unbounded} \textbf{Unbounded complexes.} The category $\Cochains{\A}$ of \textit{unbounded} cochain complexes of $\A$ is
also a descent category with weak equivalences, simple functor, $\mu$ and $\lambda$ defined as in the bounded case provided axiom (S4) holds. For instance, this is the case when $\A=R$-modules.
\end{example}

\begin{example}\label{example1simplicial} \textbf{Simplicial model categories {$[$Rod1, Theorem 3.2$]$}.} The subcategory of fibrant objects $\mc{M}_f$ of
a model category $\mc{M}$ is a descent category where $\mrm{E}$ is the class of weak equivalences of $\mc{M}$ and the
simple functor is the Bousfield-Kan homotopy \cite{BK} limit, $\holim :\Dl\mc{M}_f\longrightarrow\mc{M}_f$, as defined in \cite{Hir}. If $\mc{M}$ is a simplicial model category, the homotopy limit of a cosimplicial object $X$ is the end of the bifunctor
$X^{\mrm{N}(\Dl\downarrow \cdot)}:\Dl^{op}\times\Dl\longrightarrow\mc{M}_f$, $(n,m)\mapsto (X^m)^{\mrm{N}(\Dl\downarrow n)}$, that is,
$$
\holim X=\int_n (X^n)^{\mrm{N}(\Dl\downarrow n)} \ .
$$

Morphisms $\mu$ and $\lambda$ are easily defined using that a functor $F:\mc{B}\longrightarrow\mc{C}$ induces a natural map
$ \holim_{\mc{C}}X\longrightarrow\holim_{\mc{B}}F^\ast X$.

Two particular instances of this example are relevant when talking about sheaf cohomology theories. First,
the category $\Sset_f$ of pointed Kan complexes, with weak equivalences the weak homotopy equivalences.
Secondly, the category $\Sp_f$ of pointed fibrant
spectra, as defined in \cite[5.2]{Th}. The weak equivalences for the descent structure are then the
\textit{stable weak equivalences}; that is, morphisms of
spectra inducing bijections in all homotopy groups.
\end{example}

\begin{example}\label{examplefilteredcomplexes} \textbf{Filtered complexes.} Denote by $\textbf{F}\Cochainsp{\A}$ the category of
filtered complexes, with objects the pairs $(A,\mrm{F})$ where $A$ is in $\Cochainsp{\A}$ and $\mrm{F}$ is a decreasing filtration of $A$.
Given $r\geq 0$, consider the class $\mrm{E}_r$ of weak equivalences given by the $\mrm{E}_r$-quasi-isomorphisms of $\textbf{F}
\Cochainsp{\A}$, that is, morphisms of filtered complexes such that the induced morphism
between the $\mrm{E}_{r+1}$-terms of the spectral sequences associated with the filtrations is an isomorphism.

It holds that $(\textbf{F}\Cochainsp{\A},\mrm{E}_r)$
is a descent category with simple functor $(\mbf{s},\delta_r):\Dl\textbf{F}\Cochainsp{\A}\rightarrow \textbf{F}\Cochainsp{\A}$
defined as $(\mbf{s},\delta_r)(A,\mrm{F})=(\mbf{s}(A),\delta_r(\mrm{F}))$ where
$$\delta_r(\mrm{F})^k(\mbf{s}(A)^n) = \ds\bigoplus_{i+j=n}\mrm{F}^{k-ri}A^{i,j} \quad ,$$
and with natural transformations $\lambda$ and $\mu$ given at the level of complexes by those of $\Cochainsp{\A}$.

If $r=0$, note that an $\mrm{E}_0$-isomorphism is the same thing as a graded quasi-isomorphism. Also,
$(\mbf{s},\delta_0)(A,\mrm{F})$ is just $(\mbf{s}(A),\mbf{s}(\mrm{F}))$. The fact that this is a simple functor for
$(\textbf{F}\Cochainsp{\A},\mrm{E}_0)$ is an easy consequence of the transfer lemma applied to the graded functor
$\mrm{Gr}:\textbf{F}\Cochainsp{\A}\rightarrow \Cochainsp{\A}^{\mathbb{Z}}$.

To treat the general case, consider the decalage filtration functor $Dec:\textbf{F}\Cochainsp{\A}\rightarrow\textbf{F}\Cochainsp{\A}$,
$(A,\mrm{F})\mapsto(A,Dec\mrm{F})$, where $(Dec\mrm{F})^kA^n = \mrm{ker}\{d:\mrm{F}^{k+n}A^n\rightarrow \mrm{F}^{k+n}A^{n+1} /\mrm{F}^{k+n+1}A^{n+1}\}$.
Since  $Dec \, (\mbf{s},\delta_{r+1})= (\mbf{s},\delta_r)\, Dec$, by applying the transfer lemma inductively, we can conclude that $(\mbf{s},\delta_r)$ is a simple functor for $(\textbf{F}\Cochainsp{\A},\mrm{E}_r)$, for each $r\geq 0$.
\end{example}

\subsection{Cartan-Eilenberg categories}\label{CEsection}

\subsubsection{} Cartan-Eilenberg categories are a new approach to homotopical algebra developed in \cite{GNPR1}. They use, we believe, a minimum amount of data in order to derive functors, so its conditions can be fulfilled by a wider class of categories, as we are going to show.

\begin{definition}\label{fibrant} Let $(\C, \Scal, \W)$ be a category with two classes $\Scal$ and $\W$ of distinguished morphisms, called respectively
{\it strong\/} and {\it weak equivalences\/}, and such that ${\Scal} \subset \overline{\W}$. An object $M$ of $\C$ is called {\it Cartan-Eilenberg fibrant\/}, \textit{CE-fibrant} for short, if for each weak equivalence $w: Y \longrightarrow X \in \W$ and every morphism $f \in \loc{C}{S}$, there is a unique morphism $g \in\loc{C}{S}$ making the following triangle commutative:
$$
\xymatrix{
Y \ar[r]^w \ar[d]_-f    &   X  \ar@{.>}[dl]^g  \\
M                       &                     }
$$
\end{definition}

\begin{remark} Here $\overline{\W}$ denotes the \emph{saturation} of $\W$. Classes $\Scal$ and $\W$ of strong and weak equivalences considered later in the study of sheaves are saturated, i.e. $\overline{\Scal}=\Scal$ and $\overline{\W}=\W$. In this case, Whitehead's theorem holds: a weak equivalence between CE-fibrant objects is a strong one.
\end{remark}

\subsubsection{}\label{defCEcategory} A right {\it CE-fibrant model\/} of an object $X$ of $\C$ is a morphism $w : X \longrightarrow M$ of $\loc{C}{S}$ that becomes an isomorphism in $\loc{C}{W}$, and such that $M$ is CE-fibrant. If $X$ admits a CE-fibrant model, it is unique up to unique isomorphism of $\loc{C}{S}$.

\begin{definition} A category with strong and weak equivalences $(\C, \Scal, \W)$ is called a {\it right Cartan-Eilenberg category\/}, or {\it CE-category\/} for short, if each object $X$ of $\C$ has a CE-fibrant model.
In this case, we will also say that $\C$ \emph{has enough CE-fibrant models}.
\end{definition}

\begin{example} If $\C$ is a Quillen model category and $\Scal , \W$ are the classes of its right homotopy equivalences and weak equivalences,
respectively, then $(\C_c, \Scal, \W)$ is a right Cartan-Eilenberg category. Here $\C_c$ is the full subcategory of Quillen cofibrant
objects. In this case, every Quillen fibrant object is CE-fibrant, but
the converse needs not be true: by its very definition, CE-fibrant objects are homotopically invariant, while Quillen fibrant objects are not.
\end{example}

\begin{remark} So, CE-categories naturally include Quillen model ones and the inclusion is \lq\lq strict" in the sense that, for instance, the class $\Scal$ must not be any class of \lq\lq homotopy equivalences". This is particularly important for us because, in the case of sheaves, the global equivalences \emph{cannot} indeed be the homotopy equivalences of any Quillen model structure, as shown in \cite{GNPR2}. Since these global equivalences are such a natural ingredient for sheaves, this seems to be significant. Global equivalences are needed, for instance, to talk about sheaves satisfying Thomason descent, which are  precisely CE-fibrant models, to close the circle.
\end{remark}

\subsubsection{} In CE-categories, the derivability criterion of functors reads as follows (see \cite[3.2.1]{GNPR1}).

\begin{proposition}\label{existenciaDerivado} Let $(\C, \Scal , \W)$ be a Cartan-Eilenberg category and $F: \C \longrightarrow \D$ a functor such that $F(s)$ is an isomorphism for every strong equivalence $s \in \Scal$. Then $F$ has a right derived functor $\mathbb{R}F : \loc{C}{W} \longrightarrow \D$ whose value on objects may be computed as $\mathbb{R}F(X) = F(M)$, where $M$ is a fibrant model of $X$.
\end{proposition}

\subsubsection{} In the CE-categories considered later on, the CE-fibrant model of an object $X$ will be functorial in the sense of \cite[2.5]{GNPR1}: what we call a \emph{resolvent functor}. One of the advantages of having a resolvent functor is that, if $\C_{\mrm{fib}}$ denotes the full subcategory
of $\C$ of CE-fibrant objects, there is an equivalence of categories (\cite[Proposition 2.5.3(2)]{GNPR1})
$$
\xymatrix{ \C_{\mrm{fib}} [\mc{S}^{-1}] \ar@<0.5ex>[r]^-{\mrm{i}}_{\sim}  &  \C[\W^{-1}]
\ar@<1ex>[l]^-{R}
}$$

\section{Categories of sheaves}

We recall some general definitions and results about sheaves of sets on a Grothendieck site.. Our main objective is to point out formulas (3.1.3)  and (3.1.4) for stalks and skyscraper sheaves, respectively. Then we observe that these formulas still make sense for sheaves with values in any category with filtered colimits and arbitrary products, and that they do indeed form a pair of
adjoint functors. The associated triple gives us the cosimplicial Godement resolution.

We also show that the category of sheaves with values in a descent category inherits a natural descent structure, which will be used repeatedly in the rest of the paper.

\subsection{Sheaves of sets}

\subsubsection{} Let $\X$ be a category.  Let $\widehat{\X} =  \PrSh{\X}{\Set}$ denote the category of presheaves on $\X$ with values in the category of sets $\Set$. By the Yoneda embedding, every object $U\in \X$ can be thought of as the representable presheaf $\y U = \X(-, U) \in \widehat{\X}$.

\subsubsection{} If $\X$ is a \emph{Grothendieck site},  $\widetilde{\X}=\Sh{\X}{\Set}$ will denote the full subcategory of $\PrSh{\X}{\Set}$ whose objects are sheaves.

Sheaves may be characterized by the following property (see \cite{McLM}, page 122): a presheaf $\F \in \widehat{\X}$ is a sheaf if and only if for every object $U\in \X$ and every cover $S = \left\{ U_\alpha \longrightarrow U \right\}$ of $U$, the diagram
\begin{equation}\label{CondicionHaz0}
\xymatrix{
{\F (U)} \ar[r]   &   {\prod_{\alpha} \F (U_\alpha) }\ar@<0.5ex>[r] \ar@<-0.5ex>[r] &
{\prod_{\alpha\beta} \F (U_{\alpha\beta})}
}
\end{equation}
is an equalizer of sets. Here the second product ranges over all composable pairs $U_{\alpha\beta} \longrightarrow U_\alpha$, $U_\alpha \longrightarrow U $ with $U_\alpha \longrightarrow U \in S$ (hence also its composition $U_{\alpha\beta} \longrightarrow U$ belongs to $S$). It follows that a functor of presheaves that commutes with limits will send sheaves to sheaves.

\subsubsection{} Let $f: \X \longrightarrow \Y$ be a \emph{morphism of sites}; that is, a functor between the underlying
categories going in the opposite direction $f^{-1} : \Y \longrightarrow \X$ which is \emph{continuous}. This means that the \emph{direct image} functor
$f_* : \widehat{\X} \longrightarrow \widehat{\Y}$, $\F \mapsto \F \circ f^{-1}$, restricts to a functor between sheaves
$f_* : \widetilde{\X} \longrightarrow \widetilde{\Y}$.

\subsubsection{} Recall that a \textit{point} of a site $\X$ is by definition a pair of adjoint functors $x=(x^\ast,x_{\ast})$
$$
\xymatrix{  {\widetilde{\X} } \ar@<0.5ex>[r]^-{x^*}  &  {\Set}
\ar@<0.5ex>[l]^-{x_*}
}
\quad , \qquad \qquad \Set (x^*\F , D) = \widetilde{\X} (\F , x_*D)
$$
such that $x^*$ commutes with finite limits. The right adjoint $x_* : \Set \longrightarrow \widetilde{\X}$ gives for every set
$D $ the so called \emph{skyscraper sheaf} $x_*D $ of $D$ at the point $x$. The left adjoint $x^* : \widetilde{\X} \longrightarrow \Set$
gives for every sheaf $\F$ the \emph{fibre} or \emph{stalk} $x^*\F = \F_x$ of $\F$ at $x$.

The following \lq\lq computational" formulas for $x^*$ and $x_*$ are for us of utmost importance, since they allow us to extend them for our categories of coefficients $\D$. First, we have a canonical and functorial isomorphism
\begin{equation}\label{fibra}
x^*\F = \F_x = \dirlim_{(U,u)} \F (U) \ ,
\end{equation}
where $(U,u)$ runs over the opposite category of neighbourhoods of $x$ (\cite{SGA4}, expos\'{e} IV, 6.8). This colimit is a \emph{filtered} one.

For a set $D\in \Set$, the sheaf $x_*D$ also admits the following description: for $U \in \X$,
\begin{equation}\label{gratacels}
(x_* D ) (U) = \prod_{u \in x^*(\y U)} D_u \ ,
\end{equation}
where $D_u = D$ for all $u \in x^*(\y U)$.

\subsubsection{} A site ${\X}$ is said to have \emph{enough points} if there exists
a set $X$ formed by points of $\X$ such that a morphism
$f$ in $\widetilde{\X}$ is an isomorphism if and only if $x^*(f)$ is a bijection for all $x \in X$.

Given a family $X$ of enough points of $\X$, consider $X$ as a discrete category with only identity morphisms. Then, there is an
adjoint pair of functors
$$
\xymatrix{ {\widetilde{\X} } \ar@<0.5ex>[r]^-{p^*}  &   {\Set^X}
\ar@<0.5ex>[l]^-{p_*}
}
\quad , \qquad \qquad \Set^X (p^*\F , D) = \widetilde{\X} (\F , p_*D)
$$
defined, for $\F \in \widetilde{\X}$ and $D = (D_x)_{x\in X}\in \Set^X $, by
\begin{equation}\label{adjuncio0}
p^*\F =  (\F_x)_{x\in X} \qquad \text{and} \qquad p_*D = \prod_{x\in X} x_* (D_x)  \ .
\end{equation}

\subsection{Sheaves with general coef{f}{i}cients.}

\subsubsection{} Let  $\X $ and $\D$ be  categories. Let $\PrSh{\X}{\D}$ denote the category of presheaves on $\X$ with values in $\D$. If $\X$ is a Grothendieck site, let $\Sh{\X}{\D}$ denote the full subcategory of $\PrSh{\X}{\D}$ whose objects
are sheaves. If $\D$ has products, then $\F$ is a sheaf if and only if, for every object $U$ and every covering $S$ of $U$, diagram (\ref{CondicionHaz0}) is an equalizer in $\D$ (\cite{McL}, V.4.1).

We can also define the direct image functor for presheaves with values in an arbitrary category $\D$: if $f: \X \longrightarrow \Y$ is a morphism of sites, $f_* : \PrSh{\X}{\D} \longrightarrow \PrSh{\Y}{\D}$ is defined on objects
$\F \mapsto \F \circ f^{-1}$ and we have the following elementary result.

\begin{lemma}
Let $f : \X \longrightarrow \Y$ be a morphism of sites and $\F \in \Sh{\X}{\D}$. Then $f_*\F \in \Sh{\Y}{\D}$.
\end{lemma}

\subsubsection{} Next we construct a natural, objectwise, descent structure for sheaves on a site $\X$ with values in a descent category $(\D, \mathrm{E})$.

\begin{definition} Let $\D$ be a category with a distinguished class of morphisms $\mathrm{E}$. We will say that a morphism of
(pre)sheaves $f: \F \longrightarrow \G$ with values in $\D$ is a {\it global equivalence\/} if for any object $U \in \X$ the
morphism $f(U) : \F (U) \longrightarrow \G (U)$ is in $\mathrm{E}$. We will denote by $\Scal$ the class of global equivalences.
\end{definition}

\begin{proposition}\label{descensohaces}
Let $(\D, \mathrm{E})$ be a descent category which is assumed to be closed under products. If $\X$ is a
Grothendieck site, then $\Sh{\X}{\D}$ is a descent category with $\simple$ defined objectwise
$\simple(\F^\bullet )(U) = \simple (\F^\bullet(U)) $ and the class of weak equivalences being the global ones.
\end{proposition}

\begin{proof} $\PrSh{\X}{\D}$ inherits from $(\D, \mathrm{E})$ an objectwise descent structure in a natural way: define $\mbf{s}:\Simpl\PrSh{X}{\D}\longrightarrow\PrSh{X}{\D}$  just as $(\mbf{s} \F^\bullet)(U) = \mbf{s}(\F^\bullet(U))$ and the equivalences $\mathrm{E}^{\X} = \Scal$ are the global ones. Since $\simple :\Simpl\D\longrightarrow\D$ is assumed to preserve limits, it follows that if $\F^\bullet$ is a cosimplicial sheaf, then $\simple(\F^\bullet)$ is a sheaf. Hence, $\simple$ restricts to a functor defined between the categories of sheaves $\simple : \Simpl\Sh{\X}{\D}\longrightarrow\Sh{\X}{\D}$.

On the other hand, the subcategory $\Sh{\X}{\D}$ is closed under
products. As $\Sh{\X}{\D}$ is a full subcategory of
$\PrSh{\X}{\D}$, then the natural transformations $\lambda$ and
$\mu$ for $\PrSh{\X}{\D}$ are also natural transformations in
$\Sh{\X}{\D}$. Thus, it is clear that the hypotheses of the Transfer Lemma \ref{transferLema} are verified taking as $\psi$ the inclusion functor $ \Sh{\X}{\D}\longrightarrow \PrSh{\X}{\D}$.
\end{proof}

\subsection{The cosimplicial Godement resolution}

In this section we show how the classical cosimplicial Godement resolution makes sense for sheaves with values in categories with filtered colimits and arbitrary products.

\subsubsection{} If $\D$ is a category closed under products and filtered colimits and $x$ is a point of the site $\X$, then $x^*:\Sh{\X}{\D}\longrightarrow \D$ and $x_*:\D\longrightarrow \Sh{\X}{\D}$ may be defined by the formulas (\ref{fibra}) and (\ref{gratacels}). That is, for $\F \in \Sh{\X}{\D}$ and $D \in \D$, write
$$x^*\F = \F_x = \dirlim_{(U,u)} \F (U) \ \qquad \text{and} \qquad \ (x_* D ) (U) = \prod_{u \in x^*(\y U)} D_u \ ,
$$
where $(U,u)$ runs over the opposite category of neighbourhoods of $x$ and $D_u = D$ for all $u$.

And we need nothing else:

\begin{proposition} Let $\D$ be a category closed under products and filtered colimits, and let $x$ be a point of the site $\X$. Then, the functors $x_*$, $x^*$ define a pair of adjoint functors
$$
\xymatrix{ {\Sh{\X}{\D}}  \ar@<0.5ex>[r]^-{x^*}  &  {\D}
\ar@<0.5ex>[l]^-{x_*}
}
\quad , \qquad \qquad \D (x^*\F , D) = \Sh{\X}{\D} (\F , x_*D) \ .
$$
\end{proposition}

\subsubsection{} Consequently, for a given set $X$ of enough points of the site $\X$, formulas (\ref{adjuncio0}) also make sense for coefficients
in $\D$ and define a pair of adjoint functors
$$
\xymatrix{ {\Sh{\X}{\D} } \ar@<0.5ex>[r]^-{p^*}  & {\D^X}
\ar@<0.5ex>[l]^-{p_*}
}
\quad , \qquad \qquad \D^X (p^*\F , D) = \Sh{\X}{\D} (\F , p_*D) \ .
$$
In case where $\X$ is the site associated with a topological space $X$, we may take as a set of enough points the underling set of $X$, and
the resulting adjoint pair $(p^*,p_*)$ agrees with the one induced by the continuous map $p:X_{\mrm{dis}}\longrightarrow X$, where $X_{\mrm{dis}}$ is $X$ with the discrete topology.

\subsubsection{} Let $\mathbf{T}=(T, \eta , \nu)$ denote the triple associated with this adjoint pair of functors. Its underlying
functor $T=p_*p^* : \Sh{\X}{\D} \longrightarrow \Sh{\X}{\D}$ is given by
$$
T(\F)\, (U) = \prod_{x\in X} (x_*x^*\F)(U) = \prod_{x\in X} \prod_{u \in x^*(\y U)} (x^*\F)_u \ ,
$$
where $(x^*\F)_u = x^*\F$ for all $u\in x^*(\y U)$. The standard construction associated with the triple $\mathbf{T}= (T, \eta , \nu)$ gives a cosimplicial object
$G^\bullet (\F)\in \Dl\Sh{\X}{\D}$ with $G^p (\F) = T^{p+1}(\F) $. The resulting functor
$$
G^\bullet : \Sh{\X}{\D} \longrightarrow \Simpl\Sh{\X}{\D}
$$
is nothing else than the {\it canonical cosimplicial resolution of Godement\/} (\cite{Go}). For every sheaf $\F$, the natural
transformation $\eta : \id \longrightarrow T$ defines a coaugmentation
$\F \longrightarrow G^\bullet (\F)$, that we also denote by $\eta_{\F}$.

We will repeatedly use the following well-known property of $G^{\bullet}(\F)$, that holds for each cosimplicial
construction associated with an adjoint pair.

\begin{lemma}\label{DegeneracionExtraGodement} The natural coaugmentation $\eta:\mrm{id} \longrightarrow G^\bullet $
is such that $p^{\ast}(\eta_{\F})$ and $\eta_{p_{\ast}(D)}$ have an extra degeneracy
for each sheaf $\F\in\Sh{\X}{\D}$ and each object $D\in\D^X$. Hence, they are cosimplicial homotopy equivalences.
\end{lemma}

\section{Cartan-Eilenberg categories of sheaves}

In this section we study under which conditions the Godement cosimplicial resolution deserves to be called a `resolution'.
More precisely, we provide conditions equivalent to the fact that the Godement resolution of $\F$  produces a CE-fibrant model,
by means of other properties such as Thomason's descent, or a weak commutation of the simple functor with stalks.

\subsection{Cartan-Eilenberg fibrant sheaves}\label{CartanEilenbergFibrantSheaves}

\subsubsection{} Let $(\D,\mrm{E})$ be a descent category.

\begin{definition}\label{Eexactos}
We say that filtered colimits and arbitrary products in $\D$ are $\mrm{E}$-{\it exact\/} if:

\begin{enumerate}
\item for any filtered category $I$ and any natural transformation $\varphi: F \longrightarrow G$ between functors $F, G: I \longrightarrow \D$ such that $\varphi_i \in \mathrm{E}$ for every $i \in I$, we have $\dirlim_i \varphi_i \in \mathrm{E}$; and
\item for every arbitrary family of morphisms $\{\varphi_i : F_i \longrightarrow G_i\}_{i\in I}$ in $\D$ such that $\varphi_i \in \mathrm{E}$ for every $i\in I$, we have $\prod_{i \in I} \varphi_i \in \mathrm{E}$.
\end{enumerate}
\end{definition}

In what follows, the site $\X$ and the descent category $(\D,\mrm{E})$ are assumed to verify the following hypotheses:
\begin{equation}\label{hipotesis}
\mbox{\begin{tabular}{cl}
      (G0) & $\X$ is a Grothendieck site with a set $X$ of enough points.\\
      (G1) & $\D$ is closed under filtered colimits and arbitrary products.\\
      (G2) & Filtered colimits and arbitrary products in $\D$ are $\mathrm{E}$-exact.
      \end{tabular}
}
\end{equation}

\subsubsection{} In sheaf theory the notion of weak equivalence that best reflects the homotopical behaviour
of sheaves and the topology of the Grothendieck site is the one of local equivalence, defined stalkwise rather than objectwise.

\begin{definition} A morphism of sheaves $f: \F \longrightarrow \G$ is a {\it local equivalence\/} if for any
point $x \in X$ the morphism $x^{\ast}f : x^{\ast}\F \longrightarrow x^{\ast}\G$ belongs to $\mathrm{E}$. We will denote by $\W$ the class of
local equivalences of $\Sh{\X}{\D}$.
\end{definition}

The following properties of the classes of global and local equivalences are immediate consequences of assumptions
(\ref{hipotesis}).

\begin{lemma} The classes $\Scal$ and $\W$ of global and local equivalences are saturated. In addition:
\begin{enumerate}
 \item[\emph{(1)}] $p_{\ast}(\mrm{E})\subset \mc{S}$ and $\mc{W}=(p^{\ast})^{-1}\mrm{E}$.
 \item[\emph{(2)}] $T(\mc{W})\subset {\Scal } \subset \W$.
 \item[\emph{(3)}] Arbitrary products of sheaves are $\Scal$-exact.
\end{enumerate}
\end{lemma}

Since every global equivalence is a local one, $\Sh{\X}{\D}$ is then equipped with two classes of strong (global)
and weak (local) equivalences as in Section \ref{CEsection}. Therefore it makes sense to talk about CE-fibrant sheaves:

\begin{definition}
A sheaf $\F$ is {\it Cartan-Eilenberg fibrant} if for any solid diagram
$$
\xymatrix{
\G \ar[r]^w \ar[d]_-f    &   \G'  \ar@{.>}[dl]^g    \\
\F                      & }
$$
with $w$ a local equivalence and $f$ a morphism of $\loc{\Sh{\X}{\D}}{S}$, there exists a
unique morphism $g$ of $\loc{\Sh{\X}{\D}}{S}$ making the triangle commutative.
\end{definition}

\begin{example}
It is easy to see that $T(\F)$ is a CE-fibrant sheaf for any $\F$. This tells us that there is a significant number of
CE-fibrant sheaves. However, the class
$\{T(\F)\}$ fails in general to contain enough CE-fibrant models: for the case of positive complexes of abelian sheaves on a
topological space $X$, if any sheaf is locally equivalent to $T(\F)$ for some $\F$ then
$X$ would have cohomological dimension equal to zero.
\end{example}

\subsection{The hypercohomology sheaf} Next we make use of the Godement cosimplicial resolution to produce a wider class of
CE-fibrant sheaves: the one consisting of the hypercohomology sheaves.

\subsubsection{} Let $\X$ be a site and $(\D,\mrm{E})$ a descent category satisfying assumptions (\ref{hipotesis}). Recall that, by Proposition \ref{descensohaces}, $(\Sh{\X}{\D},\mc{S})$ is a descent category with simple functor $\mbf{s}$ defined objectwise. It follows from Theorem \emph{\cite[5.1]{Rod1}} that
$$
\xymatrix{ {\Sh{\X}{\D}[\mc{S}^{-1}] }  \ar@<0.5ex>[r]^-{c}  &  {\Sh{\X}{\Dl\D}[\mc{S}^{-1}] }
\ar@<0.5ex>[l]^-{\simple}}
$$
is an adjoint pair. In other words, $\mbf{s}(\F^{\bullet})$ is the homotopy limit of the
cosimplicial diagram of sheaves $\F^{\bullet}$. Accordingly to \cite{Th}, we make the following definition.

\begin{definition}
The {\it hypercohomology sheaf } of $\F\in \Sh{\X}{\D}$ is
the image under the simple functor of the Godement cosimplicial resolution of $\F$,
$$
\mbb{H}_{\X} (\F) = \mbf{s}G^\bullet (\F)  \ .
$$
\end{definition}

The natural coaugmentation $\eta_{\F}: \F\longrightarrow G^\bullet \F$ may be seen as a cosimplicial morphism $\eta_{\F}:c(\F) \longrightarrow G^\bullet (\F)$. The adjoint
morphism of $\eta_{\F}$ through $(c,\mbf{s})$
is then
$$
\rho_{\F} : \F \longrightarrow \mbb{H}_{\X} (\F)
$$
More explicitly, $ \rho_{\F}$ is the natural morphism of sheaves obtained as the
composition of $\mbf{s}(\eta_{\F}):\mbf{s}c(\F)\longrightarrow \mbb{H}_{\X} (\F)$ with
$\lambda_{\F}:\F\longrightarrow \mbf{s}c(\F)$, given by the descent
structure on $(\Sh{\X}{\D},\mc{S})$.

Thomason extensively studied those presheaves of fibrant spectra for which the morphism $\rho_{\F}$ is a global equivalence. Such an $\F$ was said to satisfy descent with respect to the site.

\begin{definition} A sheaf $\F\in\Sh{\X}{\D}$ is said to satisfy \textit{Thomason's descent} if
the natural morphism $ \rho_{\F} : \F \longrightarrow \mbb{H}_{\X} (\F) $ is in $\mc{S}$.
\end{definition}

\subsubsection{} To study the properties of the hypercohomology sheaf $\mbb{H}_{\X} (\F)$, we need to understand the behaviour of
the descent structure $(\mbf{s},\mu,\lambda)$ on $(\D,\mrm{E})$ with respect to the Godement pair $(p^{\ast},p_{\ast},\eta,\varepsilon)$.
On the one hand, the comparison of $\mbf{s}$ with $p_{\ast}$ presents no difficulty:
since $p_{\ast}$ only involves products and $\mbf{s}$ commutes with limits, $p_{\ast}\mbf{s}=\mbf{s}p_{\ast}$.

On the other hand, although $p^{\ast}$ does not commute in general with simple functors, we still do have a natural comparison morphism
\begin{equation}\label{psimple}
\theta_{\F^{\bullet}}: p^{\ast}\mbf{s}(\F^{\bullet})\longrightarrow \mbf{s}p^{\ast}(\F^{\bullet})
\end{equation}
defined as the adjoint through $(p^{\ast},p_{\ast})$ of
$\mbf{s}(\eta_{\F^{\bullet}}):\mbf{s}(\F^{\bullet})\longrightarrow \mbf{s}T(\F^{\bullet}) =
\mbf{s} p_{\ast} p^{\ast}(\F^{\bullet}) =
p_{\ast}\mbf{s}p^{\ast}(\F^{\bullet})$ for each cosimplicial sheaf $\F^{\bullet}$. Applying $p_{\ast}$ to (\ref{psimple}), we obtain the natural transformation
$$
\theta'_{\F^{\bullet}}:T\mbf{s}(\F^{\bullet})\longrightarrow \mbf{s}T(\F^{\bullet})
$$
adjoint to the composition
$p^{\ast}p_{\ast}p^{\ast}\mbf{s}(\F^{\bullet})\stackrel{\varepsilon_{p^{\ast}\mbf{s}(\F^{\bullet})}}{\longrightarrow} p^{\ast}\mbf{s}(\F^{\bullet})
\stackrel{\theta_{\F^{\bullet}}}{\longrightarrow } \mbf{s}p^{\ast}(\F^{\bullet})$.

The next technical lemma summarizes the compatibility relations between the descent structure on $(\D,\mrm{E})$
and the Godement pair which will be needed later on.

\begin{lemma} The diagrams
\begin{equation}\label{compatible}
\xymatrix{
 \mbf{s}(\F^{\bullet}) \ar[rd]_{\eta_{\mbf{s}(\F^{\bullet})}} \ar[rr]^{\mbf{s}(\eta_{\F^{\bullet}})} & & \mbf{s}T(\F^{\bullet})   &
T^2 \mbf{s}(\F^{\bullet}) \ar[d]^{\nu_{\mbf{s}(\F^{\bullet})}} \ar[r]^{T(\theta'_{\F^{\bullet}})} & T\mbf{s}T(\F^{\bullet})
\ar[r]^{\theta'_{T(\F^{\bullet})}} & \mbf{s}T^2 (\F^{\bullet})
 \ar[d]^{\mbf{s}(\nu_{\F^{\bullet}})}\\
     &   T\mbf{s}(\F^{\bullet}) \ar[ru]_{\theta'_{\F^{\bullet}}}  &   & T\mbf{s}(\F^{\bullet}) \ar[rr]^{\theta'_{\F^{\bullet}}} &
&  \mbf{s}T({\F}^{\bullet})
}
\end{equation}

are commutative in $\Sh{\X}{\D}$. In addition, the diagrams below are commutative, respectively, in $\mc{D}^{X}[\mrm{E}^{-1}]$ and $\Sh{\X}{\D}[\mc{S}^{-1}]$
$$
\xymatrix@C=30pt{
p^{\ast}(\F) \ar[r]^-{\lambda_{p^{\ast}(F)}} \ar[rd]_{p^{\ast}(\lambda_{\F})} & \mbf{s}c p^{\ast}(\F) = \mbf{s}p^{\ast} c(\F)  &
T(\F) \ar[r]^-{\lambda_{T(F)}} \ar[rd]_{T(\lambda_{\F})} & \mbf{s}c T (\F) = \mbf{s}Tc(\F)  \\
& p^{\ast}\mbf{s}c(\F) \ar[u]_{\theta_{c(\F)}} &  & T\mbf{s}c(\F) \ar[u]_{\theta'_{c(\F)}}
}
$$
\end{lemma}

\begin{proof} The commutativity of diagrams (\ref{compatible}) may be checked by an easy adjunction
argument. On the other hand,
note that since $p_{\ast}(\mrm{E})\subset \mc{S}$ and $p^{\ast}(\mc{S})\subset \mrm{E}$,
$$
\xymatrix{
{ \Sh{\X}{\D}[\mc{S}^{-1}]  }  \ar@<0.5ex>[r]^-{p^*}  &  {\mc{D}^{X}[\mrm{E}^{-1}] }
\ar@<0.5ex>[l]^-{p_*}
}
$$
is again an adjoint pair after localizing. The image of $\theta_{\F^{\bullet}}: p^{\ast}\mbf{s}(\F^{\bullet})\longrightarrow \mbf{s}p^{\ast}(\F^{\bullet})$
in $\mc{D}^{X}[\mrm{E}^{-1}]$ is then also the canonical morphism induced by the adjunction $(p^{\ast},p_{\ast})$
at the localized level. But $\mbf{s}:\Dl\mc{D}^{X}[\mrm{E}^{-1}]\longrightarrow \mc{D}^{X}[\mrm{E}^{-1}]$ is right adjoint
to the constant functor, and it follows formally that $\theta_{\F^{\bullet}}: p^{\ast}\mbf{s}(\F^{\bullet})\longrightarrow \mbf{s}p^{\ast}(\F^{\bullet})$
coincides with the adjoint morphism through $(c,\mbf{s})$ of the natural map
$p^{\ast}(\epsilon_{\F^{\bullet}}) : p^{\ast}c\mbf{s}(\F^{\bullet})=c\,p^{\ast}\mbf{s}(\F^{\bullet})\longrightarrow p^{\ast}(\F^{\bullet})$,
where $\epsilon_{\F^{\bullet}}$ is the unit of $(c,\mbf{s})$.
This in turn implies that the two remaining diagrams are commutative as claimed.
\end{proof}

\subsubsection{} The properties of the natural comparison transformation (\ref{psimple}) imply
the following lemma, which will be of use for the next results.

\begin{lemma}\label{retracto}
Assume that $\X$ and $(\D, \mathrm{E})$ satisfy hypotheses
\emph{(\ref{hipotesis})}. If so, for any sheaf  $\F\in \Sh{\X}{\D} $ the following conditions hold:
\begin{enumerate}
 \item[\emph{(1)}] $T(\rho_{\F}) : T(\F) \longrightarrow T\mbb{H}_{\X} (\F)$ has a natural
section in $\Sh{\X}{\D}[\mc{S}^{-1}]$.
\item[\emph{(2)}] $\rho_{\mbb{H}_{\X} (\F)} : \mbb{H}_{\X} (\F) \longrightarrow \mbb{H}_{\X} (\F)^2 =\mbb{H}_{\X}\mbb{H}_{\X} (\F)$ has a natural
section in $\Sh{\X}{\D}[\mc{S}^{-1}]$.
\end{enumerate}
\end{lemma}

\begin{remark}
In the proof of the above lemma and other results of this section, we are going to work with bicosimplicial objects.
A bicosimplicial object $\{Z^{n,m}\}_{n,m}$ will be denoted by $Z^{\circ, \bullet}$, or by $Z^{\bullet, \circ}$, to distinguish both cosimplicial indexes.
To such a bicosimplicial object we may apply the simple functor with respect to $\circ$ or with respect to $\bullet$. The resulting
cosimplicial objects will be denoted, respectively, by $\mbf{s}_{\circ} Z^{\circ, \bullet}$ and $\mbf{s}_{\bullet} Z^{\circ, \bullet}$.
\end{remark}

\begin{proof} Given a sheaf $\F\in\Sh{\X}{\D}$, let us exhibit a natural morphism $\sigma_{\F}: T\mbb{H}_{\X} (\F)\longrightarrow T(\F)$ in $\loc{\Sh{\X}{\D}}{S}$ such that $\sigma_{\F} T(\rho_{\F})= \mrm{id}$.

Recall that the coaugmentation $\eta_{\F}:c(\F)\longrightarrow G^{\bullet}(\F)$ associated with the triple
 $\mathbf{T}= (T, \eta , \nu)$ is such that $T(\eta_{\F}):cT(\F)\longrightarrow T G^{\bullet}(\F)$ has an extra degeneracy given by the maps
$\{\nu_{T^n(\F)}:T^{n+2}(\F)\longrightarrow T^{n+1}(\F)\}_{n\geq 0}$. In particular,
$$
\nu'=\{\nu_{\F}\,\nu_{T(\F)}\,\cdots\,\nu_{T^n(\F)}\}_{n\geq 0}:T G^{\bullet}(\F)\longrightarrow cT(\F)
$$
is a natural cosimplicial morphism such that $\nu'\, T(\eta_{\F})=\id$. Then $\mbf{s}(\nu')\, \mbf{s}T(\eta_{\F})=\id$.
Define $\sigma_{\F}$ as the composition
$$
\xymatrix@C=40pt{ T\mbf{s}G^{\bullet}(\F)\ar[r]^-{\theta'_{G^{\bullet}(\F)}} & \mbf{s}TG^{\bullet}(\F) \ar[r]^-{\mbf{s}(\nu')} & \mbf{s}cT(\F) \ar[r]^-{\lambda_{T(\F)}^{-1}}&
T(\F)}  \ .
$$
Since $\theta'$ is a natural transformation, then
$ \theta'_{G^{\bullet}(\F)}\, T\mbf{s} (\eta_{\F}) = \mbf{s} T (\eta_{\F}) \, \theta'_{c(\F)}$, so
$\mbf{s}(\nu')\, \theta'_{G^{\bullet}(\F)}\, T\mbf{s} (\eta_{\F}) =  \mbf{s}(\nu')\, \mbf{s} T (\eta_{\F}) \, \theta'_{c(\F)} =
 \theta'_{c(\F)}$, and
$$
\sigma_{\F} \,  T(\rho_{\F}) =  (\,\lambda_{T(\F)}^{-1}\, \mbf{s}(\nu')\, \theta'_{G^{\bullet}(\F)} \,)\,
(\, T\mbf{s} (\eta_{\F}) \, T(\lambda_{\F}) \, ) = \lambda_{T(\F)}^{-1}\, \theta'_{c(\F)}  \, T(\lambda_{\F}) =
\lambda_{T(\F)}^{-1}\,  \lambda_{T(\F)} = \mrm{id}  \ .
$$
To see (2), note that the canonical morphism $\theta'$ may be iterated to $\theta'^{n}_{\F^{\bullet}} :
T^{n+1}\mbf{s}(\F^{\bullet})\longrightarrow \mbf{s}T^{n+1}(\F^{\bullet})$, with
$(\theta')^{n+1}_{c\F}\, T^{n+1}(\lambda_\F)=\lambda_{T^{k+1}\F}$. Explicitly, $(\theta')^{n}_{\F^{\bullet}}$ is given by the composition
$$
\xymatrix@C=35pt{
T^{n+1}\mbf{s}(\F^{\bullet})\ar[r]^-{T(\theta'^{n-1}_{\F^{\bullet}})} & T\mbf{s}T^{n}(\F^{\bullet}) \ar[r]^-{\theta'_{T^{n}(\F^{\bullet})}}
& \mbf{s}T^{n+1}(\F^{\bullet})  \ .
}
$$
Since the diagrams (\ref{compatible}) are commutative, it follows that
$(\theta')^{\circ}:G^{\circ}\mbf{s}_{\bullet}(\F^{\bullet})\longrightarrow
\mbf{s}_{\bullet}G^{\circ}(\F^{\bullet}) $ is a cosimplicial morphism.
In addition, using the fact that $\theta'_{\F^{\bullet}}\, \eta_{\simple_{\bullet}(\F^{\bullet})}
= \simple_{\bullet}(\eta_{\F^{\bullet}})$ it may be proved by induction that $\theta'^{\circ}_{\F^{\bullet}}\,\eta^{\circ}_{\mbf{s}_{\bullet}(\F^{\bullet})}
=  \mbf{s}_{\bullet}(\eta^{\circ}_{\F^{\bullet}})$. In particular, for $\F^{\bullet}= G^{\bullet}(\F)$ we obtain the commutative
diagram of cosimplicial sheaves
$$
\xymatrix{
 c^{\circ}\mbf{s}_{\bullet}G^{\bullet}(\F) \ar[rd]_{\eta^{\circ}_{\mbf{s}_{\bullet}G^{\bullet}(\F)}} \ar[rr]^{\mbf{s}_{\bullet}(\eta^{\circ}_{G^{\bullet}(\F)})}
& & \mbf{s}_{\bullet}G^{\circ}G^{\bullet}(\F)   \\
     &   G^{\circ}\mbf{s}_{\bullet}G^{\bullet}(\F) \ar[ru]_{\theta'^{\circ}_{G^{\bullet}(\F)}}  &
}  \ .
$$
Hence, applying the simple functor we deduce that $\mbf{s}_{\circ}(\theta'^{\circ}_{G^{\bullet}(\F)}) \, \mbf{s}_{\circ}(\eta^{\circ}_{\mbf{s}_{\bullet}G^{\bullet}(\F)})
= \mbf{s}_{\circ} \mbf{s}_{\bullet}(\eta^{\circ}_{G^{\bullet}(\F)}) $. Assume it proved that $\mbf{s}_{\circ} \mbf{s}_{\bullet} (\eta^{\circ}_{G^{\bullet}(\F)})\in\mc{S}$.
In this case,
$$
\phi= \mbf{s}_{\circ} \mbf{s}_{\bullet}(\eta^{\circ}_{G^{\bullet}(\F)}) \, \lambda_{\mbf{s}_{\bullet}G^{\bullet}(\F)}
= \mbf{s}_{\circ}(\theta'^{\circ}_{G^{\bullet}(\F)}) \, \mbf{s}_{\circ}(\eta^{\circ}_{\mbf{s}_{\bullet}G^{\bullet}(\F)}) \,
\lambda_{\mbf{s}_{\bullet}G^{\bullet}(\F)} =  \mbf{s}_{\circ}(\theta'^{\circ}_{G^{\bullet}(\F)}) \, \rho_{\mbb{H}_{\X} (\F)}
$$
is an isomorphism of $\loc{\Sh{\X}{\D}}{S}$, so $\sigma_{\F}= \phi^{-1}\, \mbf{s}_{\circ}(\theta'^{\circ}_{G^{\bullet}(\F)})$ is a section of
$\rho_{\mbb{H}_{\X} (\F)}$. To finish, it remains to be shown that $\mbf{s}_{\circ} \mbf{s}_{\bullet}(\eta^{\circ}_{G^{\bullet}(\F)})\in\mc{S}$.
This happens if and only if $\mbf{s}_{\bullet} \mbf{s}_{\circ}(\eta^{\circ}_{G^{\bullet}(\F)})\in\mc{S}$.
For a fixed $n\geq 0$, the coaugmentation $\eta^{\circ}_{G^{n}(\F)}= \eta^{\circ}_{T^{n+1}(\F)}:c^{\circ}T^{n+1}(\F)\longrightarrow
G^{\circ}T^{n+1}(\F)$ has an extra degeneracy. Hence we infer that $\mbf{s}_{\circ}(\eta^{\circ}_{G^{n}(\F)})$ is in $\mc{S}$ for each $n\geq 0$.
But then it follows from axiom (S4) that $\mbf{s}(n\rightarrow \mbf{s}_{\circ}(\eta^{\circ}_{G^{n}(\F)}))=\mbf{s}_{\bullet} \mbf{s}_{\circ}(\eta^{\circ}_{G^{\bullet}(\F)})\in\mc{S}$
as required.
\end{proof}

\subsubsection{} The class $\mc{W}$ of local equivalences is by definition equal to $(p^{\ast})^{-1}\mrm{E}$. Below we prove that $\mc{W}=T^{-1}\mc{S}=\mbb{H}_{\X}^{-1}\mc{S}$ as well.

\begin{proposition}\label{localeq} Assume that $\X$ and $(\D, \mathrm{E})$ satisfy the hypotheses
\emph{(\ref{hipotesis})}. Then, for a morphism  $f:\F\longrightarrow \G$  of sheaves, the following conditions are equivalent:
\begin{enumerate}
 \item[\emph{(1)}] $f$ is a local equivalence.
 \item[\emph{(2)}] $T(f):T(\F)\longrightarrow T(\G)$ is a global equivalence.
 \item[\emph{(3)}] $\mbb{H}_{\X}(f):\mbb{H}_{\X} (\F)\longrightarrow \mbb{H}_{\X} (\G)$ is a global
equivalence.
\end{enumerate}
\end{proposition}

\begin{proof} (1) implies (2) since $T(\mc{W})\subset \mc{S}$. Conversely,
if $T(f)$ is a global equivalence, it is in particular a local one, so $p^{\ast}T(f)\in\mrm{E}$.
On the other hand, it follows from the triangle identities
of the adjoint pair $(p^{\ast},p_{\ast})$ that $p^{\ast}(f)$ is a retract of $p^{\ast}T(f)=p^{\ast}p_{\ast}p^{\ast}(f)$.
But $\mrm{E}$ being saturated, it is closed under retracts, and we deduce that $p^{\ast}(f)\in\mrm{E}$ as well.
But this is the same as saying that $f\in\Scal$. Therefore, (1) and (2) are equivalent.

Let us see that (2) implies (3). Assume that $T(f)\in\mc{S}$. Since
$T(\mc{S})\subset T(\mc{W})\subset \mc{S}$, then $G^{n}(f)=T^{n+1}(f)\in\mc{S}$ for all $n\geq 0$,
and it follows from (S4) that $\mbb{H}_{\X}(f)=\mbf{s}G^\bullet (f)\in\mc{S}$ as required.
Finally, if $\mbb{H}_{\X} (f)\in\mc{S}$ then also
$T\mbb{H}_{\X} (f)\in\mc{S}$. By Lemma \ref{retracto} $T(f)$ is a retract of $T\mbb{H}_{\X} (f)$, so
$T(f)\in\mc{S}$ and (2) and (3) are equivalent as well.
\end{proof}

\subsubsection{} As announced, we deduce that the hypercohomology sheaf is always CE-fibrant.

\begin{proposition}\label{GodementFibrante}
Assume that $\X$ and $(\D, \mathrm{E})$ satisfy hypotheses
\emph{(\ref{hipotesis})}. Then, for any sheaf $\F$,
$\mbb{H}_{\X} (\F)$ is a CE-fibrant sheaf.
\end{proposition}

\begin{proof} Hypothesis \ref{hipotesis} guarantee that $T(\mc{S})\subset \mc{S}$, hence $\mbb{H}_{\X}(\mc{S})\subset \mc{S}$. By Lemma \ref{retracto}, it is equipped with natural transformations $\rho: \mrm{id}\longrightarrow \mbb{H}_{\X}$ and $\sigma:\mbb{H}_{\X}^2\longrightarrow \mbb{H}_{\X}$ such
that $\sigma\,\rho=\mrm{id}$.

As a first consequence, a morphism $g:\G\longrightarrow \mbb{H}_{\X} (\F)$ of $\Sh{\X}{\D}[\mc{S}^{-1}]$ is uniquely determined by $\mbb{H}_{\X}(g)$. Indeed, from the commutative diagram
$$
\xymatrix@C=30pt{
\G \ar[r]^-{g} \ar[d]_{\rho_{\G}} & \mbb{H}_{\X} (\F) \ar[r]^-{1} \ar[d]_{\rho_{\mbb{H}_{\X} (\F)}} & \mbb{H}_{\X} (\F) \\
\mbb{H}_{\X}(\G) \ar[r]^-{\mbb{H}_{\X}(g)}  & \mbb{H}_{\X} (\F)^2 \ar[ru]_{\sigma_{\F}}
}$$
we deduce that $g=\sigma_{\F}\,\mbb{H}_{\X}(g)\,\rho_{\G}$ as claimed.
Consider now a lifting problem
$$
\xymatrix{
\G \ar[r]^w \ar[d]_-f    &   \G'      \\
\mbb{H}_{\X} (\F)                      & }
$$
where $f$ is a morphism of $\Sh{\X}{\D}[\mc{S}^{-1}]$ and $w$ is a morphism
of $\Sh{\X}{\D}$ that is a local equivalence.
Since $\mbb{H}_{\X}(\mc{W})\subset \mc{S}$, given two solutions $g,g':\G'\longrightarrow \mbb{H}_{\X} (\F)$ of this lifting problem, we would have
$\mbb{H}_{\X}(g)=\mbb{H}_{\X}(f)\, (\mbb{H}_{\X}(w))^{-1}= \mbb{H}_{\X}(g')$. Hence
$g=g'$, and we need only see that there is at least one lifting for the above diagram. But $g=\sigma_{\F}\,\mbb{H}_{\X}(f)\, (\mbb{H}_{\X}(w))^{-1}\,\rho_{\mbb{H}_{\X} (\F)}$
is easily seen to satisfy $g\,w = f$, so we are done.
\end{proof}

\subsection{Characterization}
In view of the last proposition, we conclude that if for any sheaf  $\eta_\F : \F \longrightarrow \mathbb{H}_\X (\F)$ were in
${\W}$, then  $(\Sh{\X}{\D},\Scal, \W)$ would be a Cartan-Eilenberg category with $(\mbb{H}_{\X},\rho)$ as a resolvent functor. Below we show that this fact is indeed equivalent to two other conditions: one of them is Thomason's descent property for hypercohomology sheaves, while the other one consists of a weak commutation between the simple functor and stalks.

\subsubsection{} Let us state precisely what we mean by the later condition.

\begin{definition} Let $\X$ be a Grothendieck site and $(\D, \mathrm{E})$ a descent category. We say that the simple functor  {\it commutes weakly with stalks\/} if for each sheaf $\F$ the map $\theta_{G^\bullet \F} : p^{\ast}\mbb{H}_{\X} (\F) =  p^{\ast} \simple G^\bullet (\F) \longrightarrow \simple p^{\ast} G^\bullet (\F)
$ in (\ref{psimple}) belongs to $\mrm{E}$.
\end{definition}

Equivalently, $\mbf{s}$ commutes weakly with stalks if for each point
$x\in X$ the canonical map $\theta_{G^\bullet \F}(x) : (\simple G^\bullet \F)_x \longrightarrow \simple (G^\bullet \F)_x$ is a weak equivalence.

\subsubsection{} We can now state and prove our first main result.

\begin{theorem}\label{CartanEilenbergSheaves} Let $\X$ be a Grothendieck site and $(\D, \mathrm{E})$ a descent
category satisfying the hypotheses \emph{(\ref{hipotesis})}. Then, the
following statements are equivalent:
\begin{enumerate}
 \item[\emph{(1)}] $(\Sh{\X}{\D}, \Scal, \mc{W})$ is a right Cartan-Eilenberg category and for every sheaf $\F \in \Sh{\X}{\D}$,
$\rho_\F : \F \longrightarrow \mbb{H}_{\X} (\F)$ is a CE-fibrant model.
 \item[\emph{(2)}] For every sheaf $\F\in \Sh{\X}{\D}$, $\rho_\F : \F\longrightarrow \mbb{H}_{\X} (\F)$ is in $\mc{W}$.
 \item[\emph{(3)}] The simple functor commutes weakly with stalks.
 \item[\emph{(4)}] For every sheaf $\F\in \Sh{\X}{\D}$, $\mbb{H}_{\X} (\F)$ satisfies Thomason's descent; that is, $\rho_{\mbb{H}_{\X} (\F)} :
\mbb{H}_{\X} (\F) \longrightarrow \mathbb{H}^2_\X (\F)$ is in $\Scal $.
\end{enumerate}
\end{theorem}

\begin{definition} We say that a descent category $(\D, \mathrm{E})$ is \textit{compatible with the site} $\X$
if the equivalent conditions of this theorem are satisfied.
\end{definition}

\begin{remark} As we will see in the examples, this is not necessarily the case for general $\X$ and $(\D,\mrm{E})$. Furthermore, it may happen that $(\Sh{\X}{\D},\Scal, \W)$ is indeed a Cartan-Eilenberg category,
but the CE-fibrant model of a sheaf $\F$ does not agree with $\mbb{H}_{\X} (\F)$ in general.
However, this does not pose much of a problem, and these drawbacks only occur when $\X$ is \lq\lq cohomologically big": a suitable finite cohomological dimension hypothesis on $\X$ ensures that the hypercohomology sheaf $\mbb{H}_{\X} (\F)$ is always a (CE-fibrant) model for $\F$.
\end{remark}

\begin{proof}[{\bf Proof of Theorem \ref{CartanEilenbergSheaves}}] By Proposition \ref{GodementFibrante} we know that $\mbb{H}_{\X} (\F)$ is CE-fibrant for any sheaf $\F$.
Hence, the equivalence between (1) and (2) is clear. Let us see that (2) and (3) are
equivalent. On the one hand, by definition, (2) holds if and only if
$p^{\ast}(\rho_{\F})$ is in $\mrm{E}^{X}$ for any sheaf $\F$.
On the other hand, the cosimplicial Godement resolution is such that the coaugmentation $p^{\ast}\eta_{\F}: cp^{\ast}(\F)\longrightarrow  p^{\ast} G^\bullet (\F)$ has an extra degeneracy. It then follows from Proposition \ref{descensohaces} that
$\simple p^{\ast}(\eta_{\F})$ belongs to $\mrm{E}$. Since $\lambda_{\G}:\G\longrightarrow\mbf{s}c(\G)$ is also in $\mrm{E}$ for any sheaf $\G$, we have the
following commutative diagram in which the arrows decorated with $\sim$ are in $\mrm{E}$:
$$
\xymatrix@C=35pt@R=30pt{p^{\ast}(\F) \ar[r]^-{p^{\ast}(\lambda_{\F})}_{\rotatebox[origin=c]{0}{$\sim$}} \ar[rd]^{\lambda_{p^{\ast}(\F)}}_{\rotatebox[origin=c]{315}{$\sim$}} & p^{\ast}\simple c (\F)  \ar[r]^-{p^{\ast}\simple(\eta_{\F})} \ar[d]^{\theta_{c(\F)}} & p^{\ast} \simple G^\bullet (\F) = p^{\ast}\mbb{H}_{\X} (\F)  \ar[d]^{\theta_{G^\bullet (\F)}} \\
 & \simple p^{\ast} c(\F) \ar[r]_{\rotatebox[origin=c]{0}{$\sim$}}^-{\simple p^{\ast} (\eta_{\F})}  & \simple p^{\ast} G^\bullet (\F)
}$$
Note that the composition of the morphisms in the top row is precisely
$p^{\ast}(\rho_{\F}):p^{\ast}(\F)\longrightarrow  p^{\ast}\mbb{H}_{\X} (\F)$.
By the 2-out-of-3 property, we conclude that $p^{\ast}(\rho_{\F})$ is  in $\mrm{E}$ if and only if $\theta_{G^\bullet (\F)}$
is in $\mrm{E}$. In other words, (2) and (3) are equivalent.

To finish with, we now show that (4) and (2) are equivalent.
Because of Proposition \ref{localeq}, $\W = \mathbb{H}_X^{-1}\Scal$. Hence, $\rho_\F : \F\longrightarrow \mbb{H}_{\X} (\F)$ is
in $\mc{W}$ if and only if $\mbb{H}_{\X}(\rho_\F)$ is in $\mc{S}$. It is then enough to check that
$\rho_{\mbb{H}_{\X} (\F)}$ is in $\mc{S}$ if and only if
$\mbb{H}_{\X}(\rho_\F)$ is. As in the proof of Lemma \ref{retracto}, the iteration of $\theta'$
gives a canonical morphism of cosimplicial objects ${\theta'}^{\circ}_{\F^{\bullet}}:G^{\circ}\mbf{s}_{\bullet}(\F^{\bullet})\longrightarrow
\mbf{s}_{\bullet}G^{\circ}(\F^{\bullet})$ that makes the following diagrams commute
$$
\xymatrix@C=50pt@R=40pt{
   & \mbf{s}_{\circ} \mbf{s}_{\bullet} G^{\circ} c^{\bullet}(\F)  \ar[r]_-{\rotatebox[origin=c]{0}{$\sim$}}^-{\mbf{s}_{\circ} \mbf{s}_{\bullet} G^{\circ}(\eta^{\bullet}_{\F})}
&\mbf{s}_{\circ} \mbf{s}_{\bullet} G^{\circ} G^{\bullet}(\F)  &  \mbf{s}_{\circ} \mbf{s}_{\bullet} c^{\circ} G^{\bullet}(\F)
\ar[l]^-{\rotatebox[origin=c]{0}{$\sim$}}_-{\mbf{s}_{\circ} \mbf{s}_{\bullet} (\eta^{\circ}_{G^{\bullet}(\F)})} \ar[ld]^-{\mbf{s}_{\circ} (\eta^{\circ}_{\mbf{s}_{\bullet} G^{\bullet}(\F)})}
\\
\mbf{s}_{\circ} G^{\circ}(\F) \ar@/_3pc/[rr]_{\mbb{H}_{\X}(\rho_{\F})} \ar[r]^-{\rotatebox[origin=c]{0}{$\sim$}}_-{\mbf{s}_{\circ} G^{\circ}(\lambda_{\F}) }
\ar[ru]_{\rotatebox[origin=c]{45}{$\sim$}}^-{\mbf{s}_{\circ}(\lambda_{ G^{\circ}(\F)})}
& \mbf{s}_{\circ} G^{\circ} \mbf{s}_{\bullet}c^{\bullet}(\F) \ar[r]_-{\mbf{s}_{\circ} G^{\circ}\mbf{s}_{\bullet}(\eta^{\bullet}_{\F})}
 \ar[u]^{\rotatebox[origin=c]{90}{$\sim$}}_-{\mbf{s}_{\circ}({\theta'}^{\circ}_{c^{\bullet}(\F)})}  & \mbf{s}_{\circ} G^{\circ} \mbf{s}_{\bullet}G^{\bullet}(\F)
 \ar[u]^-{\mbf{s}_{\circ}({\theta'}^{\circ}_{G^{\bullet}(\F)})} &  \mbf{s}_{\bullet} G^{\bullet}(\F) \ar[u]^-{\rotatebox[origin=c]{90}{$\sim$}}_-{\lambda_{\mbf{s}_{\bullet} G^{\bullet}(\F)}}
 \ar[l]_-{\rho_{\mbb{H}_{\X} (\F)}}
 }
$$
Note that all the arrows decorated with $\sim$ are global equivalences: for those arrows involving $\lambda$ this is clear
(in particular this is so for $\mbf{s}_{\circ}({\theta'}^{\circ}_{c^{\bullet}(\F)})$).
We already proved that $\mbf{s}_{\circ} \mbf{s}_{\bullet} (\eta^{\circ}_{G^{\bullet}(\F)})\in\mc{S}$,
and again using an extra degeneracy argument it readily follows that
 $\mbf{s}_{\circ} \mbf{s}_{\bullet} G^{\circ} (\eta^{\bullet}_{\F})\in\mc{S}$.
Consequently, $\rho_{\mbb{H}_{\X} (\F)}\in\mc{S}$ if and only if
$\mbb{H}_{\X}(\rho_\F)\in\mc{S}$.
\end{proof}

\subsubsection{}\label{toyexample} As a toy example, let's check what our main theorem says for the case of a topological space with
just one point.

Let $X=\left\{ x \right\}$ be a topological space with just one point and with its unique possible topology; namely, its open sets
are $\mathbf{Open} (X) = \left\{ \emptyset, \left\{ x \right\}\right\}$. So, every sheaf $\F \in \Sheaf (X,\D)$ is determined by its
value on $x$: $\F (x) \in \D$. The correspondence $\phi: \Sheaf (X,\D) \longrightarrow \D$, $\F \longmapsto \F(x)$ defines an
isomorphism of categories whose inverse is $\psi: \D \longrightarrow \Sheaf (X,\D)$, $ D \longmapsto \underline{D}$, where $\underline{D}$ is the sheaf defined by $\underline{D}(x) = D$.

Next, in a \emph{sober} space such as $X$, the points of the site $\mathbf{Open} (X)$ are in a bijective correspondence with
the points of $X$ as a plain topological space. So, we have exactly one Grothendieck point; that is, a couple of adjoint functors
$x^*: \Sheaf (X, \D) \rightleftarrows  \D : x_*$, defined by $x^*(\F) = \F_x = \F (x) $ and $ (x_*D)(x) = D$. In other words, $x^* = \phi$
and $x_* = \psi$. Hence, if we identify $\Sheaf (X, \D)$ with $\D$ using $\phi$ and $\psi$, $x_*$ and $x^*$  become the identity functor
of $\C$. Hence, the Godement construction $G^\bullet : \D \longrightarrow \Simpl\D$ is simply the constant cosimplicial functor.
Applying the simple functor, we get $\mathbb{H}_X (D) = \simple G^\bullet (D) =\simple cD  \simeq D$, because
of axiom (S3) of a descent category.

This entails that every object $D$ should be fibrant with the CE-structure given on $\D$ by our main theorem. The reader can easily
check that it is so: under the identifications $\phi$ and $\psi$, classes of local and global equivalences are just $\mathrm{E}$:
$\W = \Scal = \mathrm{E}$ and, with these local and global equivalences, every descent category is a CE-category in which every object
is fibrant.

So, condition (1) of our main theorem is indeed fulfilled. The reader can check, for instance, that condition (3), the
commutation between stalks and simple functor, is also trivially fulfilled too.

\subsubsection{} The first consequence of our main theorem is the following characterization of Thomason's descent property for sheaves of spectra.

\begin{corollary}\label{Thomasondescentproperty} If $(\D, \mathrm{E})$ is compatible with the site $\X$,
then a sheaf $\F\in\Sh{\X}{\D}$ satisfies Thomason's descent if and only if it is a CE-fibrant sheaf.
\end{corollary}

\subsubsection{} The existence of an associated sheaf functor, or \textit{sheafification}, $(-)^a : \PrSh{\X}{\D}\longrightarrow\Sh{\X}{\D}$
guarantees that the homotopy theory of presheaves is the same as the homotopy theory of sheaves, because
the adjoint pair $(-)^a : \PrSh{\X}{\D}\rightleftarrows\Sh{\X}{\D}:\mbf{i}$, where $\mbf{i}$ is the inclusion functor, induces an equivalence of categories
$\loc{\PrSh{\X}{\D}}{\W} \simeq \loc{\Sh{\X}{\D}}{\W}$. Although an associated sheaf functor may not exist for $\D$, when
$(\D,\mrm{E})$ is compatible with the site $\X$ the hypercohomology sheaf may be
thought of as a `homotopical' sheafification functor. More precisely, the adjoint pair $(p^{\ast},p_{\ast})$ is also an adjoint pair
$$
\xymatrix{ {\PrSh{\X}{\D} }  \ar@<0.5ex>[r]^-{p^*}  &  {\mc{D}^{X} }
\ar@<0.5ex>[l]^-{p_*}
}
$$
and the induced triple on $\PrSh{\X}{\D}$ allows an analogous definition $\mbb{H}_{\X}(\F)=\mbf{s}G^{\bullet}(\F)$ for a presheaf $\F$, which enjoys the same properties as
in the sheaf case. In addition, $T(\F)=p_{\ast}p^{\ast}(\F)$ is a sheaf,
and so is $\mbb{H}_{\X}(\F)$.

\begin{corollary}\label{hacesiprehaceslomismoson} Let $(\D, \mathrm{E})$ be a descent category compatible with the site $\X$. Then
$$
\xymatrix{ {\PrSh{\X}{\D}[\W^{-1}]  }  \ar@<0.5ex>[r]^-{\mathbb{H}_\X}  &  { \Sh{\X}{\D}[\W^{-1}]   }
\ar@<0.5ex>[l]^-{\mbf{i}}
}
$$
are inverse equivalences of categories.
\end{corollary}

\begin{proof}By hypothesis, $\rho_{\F}:\F\longrightarrow \mbb{H}_{\X}\mbf{i}(\F)$ is in $\mc{W}$, so it is an isomorphism of
$\loc{\Sh{\X}{\D}}{\W}$ for any sheaf $\F$. It remains to be shown that if $\F$ is now a presheaf then $\rho_{\F}:\F\longrightarrow \mbf{i}\mbb{H}_{\X} (\F)$
is in ${\W}$. Since $\mbb{H}_{\X}(\F)$
is a sheaf, $\rho_{\mbb{H}_{\X}(\F)}\in{\W}$. But $\rho_{\mbb{H}_{\X}(\F)}$ is a morphism between CE-fibrant sheaves and hence belongs to $\mc{S}$.
By the same proof as in Theorem
\ref{CartanEilenbergSheaves}, we infer that $\mbb{H}_{\X}(\rho_{\F})\in\mc{S}$ as well. Again, this means that $\rho_{\F}$ is a local equivalence as required.
\end{proof}

\subsubsection{} We have seen that a descent structure on $(\D,\mrm{E})$ always induces one on $(\Sh{\X}{\D},\Scal)$ defined objectwise. We have another descent structure, though.

\begin{proposition} Assume that a descent category $(\D, \mathrm{E})$ is compatible with the site $\X$ and that filtered colimits commute with finite products
in $\D$. Then, $(\Sh{\X}{\D},\W)$ is a descent category with simple functor
$$
\simple'=\simple \,  \mbb{H}_{\X} : \Simpl \Sh{\X}{\D}\longrightarrow \Sh{\X}{\D} \ .
$$
\end{proposition}

\begin{proof} The commutation of finite products with filtered colimits guarantees that $\W\prod\W\subset \W$. The fact
that $\simple'$ is a simple functor for $(\Sh{\X}{\D},\W)$ may be proved using that $\mbb{H}_{\X}(\mc{W})\subset \mc{S}$ and
that $\mbf{s}$ is a simple functor for $(\Sh{\X}{\mc{D}},\Scal)$.
\end{proof}

It follows from the results in \cite{Rod1} that
path and loop functors may be constructed for $(\Sh{\X}{\D},\W)$ in a natural way. They give rise to  well behaved fiber sequences,
satisfying the usual properties in $\loc{\Sh{\X}{\D}}{\W}$. In particular, $\loc{\Sh{\X}{\D}}{\W}$ is a triangulated category
provided that the loop functor is an equivalence of categories.

\subsection{Derived functors for sheaves}

\subsubsection{} The second consequence of our characterization of CE-fibrant sheaves, the existence of the right derived direct image functor, follows immediately (cf. \cite[th.6 ]{Br}).

\begin{corollary}\label{existenciaderivadoimagenesdirectas}
Let $f: \X \longrightarrow \Y$ be a continuous functor of Grothendieck sites and $(\D, \mathrm{E})$ a descent category compatible with the site $\X$. Then, $f_* : \Sh{\X}{\D} \longrightarrow \Sh{\Y}{\D}$ admits a right derived functor $ \rdf f_* : \Sh{\X}{\D}[\W^{-1}] \longrightarrow \Sh{\Y}{\D}[\W^{-1}]$
given by
$$
\rdf f_* (\F) = f_* \mbb{H}_{\X} (\F) \ .
$$
\end{corollary}

\begin{proof}
In view of Theorem \ref{CartanEilenbergSheaves}  and
Proposition \ref{existenciaDerivado}, we only need to show that $f_*$ sends global equivalences to local equivalences. But this is obvious: if $\varphi : \F \longrightarrow \G \in \Scal$, then, for every object $V \in \Y$, we have $f_*(\varphi)(V) = \varphi (f^{-1}(V)) : \F( f^{-1}(V) )  \longrightarrow \G ( f^{-1}(V) ) \in \mathrm{E}$. So $f_*(\varphi)$ is also a global equivalence and hence, a fortiori, a local one.
\end{proof}

If $U$ is an object of $X$, the same proof works for the $U$-sections functor $\Gamma (U, -) : \Sh{\X}{\D}
\longrightarrow \D$ because, by definition, $\Gamma (U, \F)=\F(U)$ sends global equivalences in $\Sh{\X}{\D}$ to equivalences in
$\D$. Hence,

\begin{corollary}\label{existenciaderivadosecciones}
Let $(\D, \mathrm{E})$ be a descent category compatible with the site $\X$. Then $\Gamma (U, -) : \Sh{\X}{\D}
\longrightarrow \D$ admits a right derived functor
$\rdf \Gamma (U,-)  : \Sh{\X}{\D}[\W^{-1}] \longrightarrow \D [\mathrm{E}^{-1}]$ given by
$$\rdf \Gamma (U,\F) = \Gamma (U,\mbb{H}_{\X} (\F)) \ .$$
\end{corollary}

\subsubsection{} When $\X$ has a terminal object $X$, e.g. in case $\X$ is the site associated with a topological space $X$,
{\it sheaf cohomology} is by definition the right derived functor of the global sections functor $\Gamma (X, -) : \Sh{\X}{\D}
\longrightarrow \D$. So, under the above assumptions,
sheaf cohomology is well defined and agrees with $\Gamma (X,\mbb{H}_{\X} (\F))$.

Following \cite[4.3.6.1]{SGA4}, if the coefficient category $\D$ has limits, the notion of global sections functor
$\Gamma (\X, -) : \Sh{\X}{\D} \longrightarrow \D$ generalizes to a
general site $\X$, possibly without a terminal object, as:
$$
\Gamma(\X,\F) = \invlim_{U\in\X}\, \F(U) \ .
$$
Note that in this case $\Gamma(\X,-)$ does not necessarily send a global equivalence to a weak equivalence of $\D$.
But, being $({\D},\mrm{E})$ a descent category in which arbitrary products are $\mrm{E}$-exact, the right derived functor
of $\invlim_{\X}:\mc{D}^{\X}\longrightarrow \mc{D}$ exists, and is given by the composition of the simple functor with the cosimplicial replacement
$\mc{D}^{\X}\longrightarrow \Dl\mc{D}$ (see \cite{Rod2}). The resulting functor $\holim_{\X} : \Sh{\X}{\D} \longrightarrow \D $ sends global equivalences to weak ones; hence, it admits a right derived functor $\Sh{\X}{\D}[\mc{W}^{-1}] \longrightarrow \D[\mrm{E}^{-1}]$ that
may be seen to agree with the right derived functor of $\Gamma(\X,-)$. That is,
$\rdf \Gamma (\X, -) : \Sh{\X}{\D}[\mc{W}^{-1}] \longrightarrow \D[\mrm{E}^{-1}]$ exists and is given by
$$
\rdf \Gamma (\X, \F) =  \holim_{U\in\X}  \mbb{H}_{\X} (\F)(U) \ .
$$

\subsubsection{} Recall that when there is a sheafification functor then $\Sh{\X}{\D}$ is complete (resp. cocomplete)
when $\D$ is. A homotopical version of this fact is that when $\mbb{H}_{\X}$ is a `homotopical'
sheafification functor (that is, when $(\D,\mrm{E})$ is compatible with the site $\X$) then
$(\Sh{\X}{\D},\mc{W})$ is homotopically complete, and homotopically cocomplete provided $(\D,\mrm{E})$ is.

The key points to seeing this are that the resolvent functor $(\mbb{H}_{\X},\rho)$ is also a resolvent functor for presheaves, and that it may be lifted to diagram categories: for each small category $I$, $(\mbb{H}_{\X},\rho)$ induces objectwise a resolvent functor on $(\PrSh{\X}{\D}^I =\PrSh{\X}{\D^I},\mc{S},\mc{W})$.
This in turn implies that there is an adjunction natural in $I$
$$ \xymatrix@C=35pt@H=4pt{
\PrSh{\X}{\D}^I[\Scal^{-1}] \ar@<0.7ex>[r]^-{\mbf{id}} & \PrSh{\X}{\D}^I[\W^{-1}] \simeq \Sh{\X}{\D}^I[\W^{-1}]
\ar@<0.7ex>[l]^-{\mbb{H}_{\X}} }$$
where the right adjoint $\mbb{H}_{\X}$ is fully faithful. This natural adjunction then transfers homotopy limits
and colimits existing for $(\PrSh{\X}{\D},{\Scal})=(\D^{\X},\mrm{E}^{\X})$ to $\loc{\Sh{\X}{\D}}{\W}$. In particular $(\Sh{\X}{\D},\mc{W})$ is homotopically complete and
$$ \holim_I^{(\Sh{\X}{\D},\mc{W})} = \holim_I^{(\Sh{\X}{\D},\mc{S})} \, \mbb{H}_{\X} \ .$$

\section{Examples}

In this section we show how the above results apply to classic and not so classic examples of categories of sheaves.
More concretely, we will prove that a finite cohomological dimension assumption on the site $\X$ guarantees
its compatibility with the natural descent structures seen on categories of coefficients $\D$ such as complexes, simplicial sets and
spectra. Consequently, from the results of the previous section we conclude that for such $\X$ and  $\D$ we have:

\begin{itemize}\label{PropertiesCompat}
\item[$\bullet$] For every sheaf $\F$, the natural arrow $\rho_\F : \F \longrightarrow \mathbb{H}_\X(\F)$ is a fibrant model of $\F$.
Or, what amounts to the same, $(\Sh{\X}{\D},\Scal,\W)$ is a CE-category with resolvent functor $(\mbb{H}_{\X},\rho)$.
\item[$\bullet$] The localized category $\loc{\Sh{\X}{\D}}{\W}$ is naturally equivalent to $\loc{\Sh{\X}{\D}_{\mrm{fib}}}{\Scal}$.
\item[$\bullet$] The CE-fibrant objects of $\Sh{\X}{\D}$ are precisely those sheaves satisfying Thomason's descent.
\item[$\bullet$] Derived sections $\mbb{R}\Gamma(U,-)$ and derived direct image functor $\mbb{R}f_*$ may be computed by precomposing with
$\mbb{H}_{\X}$.
\item[$\bullet$] The hypercohomology sheaf $\mbb{H}_{\X}$ is a `homotopical' sheafification functor that gives an equivalence
$\loc{\Sh{\X}{\D}}{\W}\simeq \loc{\PrSh{\X}{\D}}{\W}$.
\end{itemize}

\subsection{Bounded complexes of sheaves}\label{example2bounded}

\subsubsection{} Consider the descent category structure on the category of uniformly bounded cochain complexes $\Cochainsp{\mc{A}}$ described in
example \ref{example1bounded}.

In this case the simple functor is
$\mbf{s}=\Tot^{\prod} = \Tot^\oplus : \Dl \Cochainsp{\mc{A}}\longrightarrow \Cochainsp{\mc{A}}$
by the boundedness assumption. The category of sheaves of uniformly bounded cochain complexes
$\Sh{\X}{ \Cochainsp{\mc{A}} }$ is a descent category where the weak equivalences are the global equivalences and the simple
functor is the total-sum functor applied objectwise: $(\Tot \F ) (U) = \Tot (\F(U))$.

It follows that $\mbf{s}$ commutes in this case with all colimits, since
it is defined degree-wise through a finite direct sum. Hence, $\mbf{s}$ commutes trivially with stalks.
Therefore, we deduce from Theorem \ref{CartanEilenbergSheaves}

\begin{theorem}\label{CEboundedcomplexes} Assume that $\A$  is an abelian category satisfying $($AB4$)^{\ast}$ and $($AB5$)$ (that is, arbitrary products and
filtered colimits exist and are exact). Then, the descent category $\Cochainsp{\A}$ is compatible with any site $\X$.
In particular, properties  \emph{\ref{PropertiesCompat}} hold for $(\Sh{\X}{\Cochainsp{\A}},\Scal,\W)$.
\end{theorem}

In this case a local equivalence $f\in\mc{W}$ is just a quasi-isomorphism of
$\Sh{\X}{\Cochainsp{\A}}=\Cochainsp{\Sh{\X}{\A}}$. On the other hand, a global equivalence $f\in\mc{S}$ is
a morphism $f:\F\longrightarrow\G$ of complexes of sheaves such that $f(U)$ is a quasi-isomorphism of $\Cochainsp{\A}$
for each object $U\in\X$.

Consequently, a functor $\mrm{F}:\Sh{\X}{\Cochainsp{\A}}\longrightarrow
\mc{C}$ sending global equivalences to isomorphisms admits a right derived functor
$\rdf \mrm{F}:\mc{D}^{\geq b}(\Sh{\X}{\A})=\Sh{\X}{\Cochainsp{\A}}[\mc{W}^{-1}]\longrightarrow \mc{C} $ given by
$\rdf \mrm{F}(\F)= \mrm{F}(\mbb{H}_{\X} (\F))$. Note that this derivability criterion does not assume the existence of enough injectives in $\A$.
Particularly, for the case $\A = R-$modules, we recover the classic construction of abelian sheaf hypercohomology
and derived direct image of sheaves constructed through canonical Godement resolutions by flasque sheaves.

\subsection{Unbounded complexes of sheaves}\label{example2unbounded}

\subsubsection{} When the boundedness assumption on complexes of sheaves is dropped, Theorem \ref{CEboundedcomplexes} is not longer true for a general site $\X$, even in the case $\A=R-$modules.

Consider the category $\Cochains{R}$ of unbounded cochain complexes of $R$-modules with the descent
structure of example \ref{example1unbounded}. In this case, the simple functor
$\mbf{s}=\Tot^{\prod} : \Dl \Cochains{R}\longrightarrow \Cochains{R}$ is an infinite product degree-wise,
and consequently it does not commute (even weakly) with filtered colimits. This in turn means that the hypercohomology sheaf $\mbb{H}_{\X} (\F)$ associated with an unbounded
complex $\F$ of sheaves of $R$-modules does not necessarily produce a CE-fibrant model for $\F$ in
$(\Sh{\X}{\Cochains{R}},\mc{S},\mc{W})$, for a general site $\X$.

\begin{example}
To illustrate this fact, consider a family $\{\F^{-k}\}_k$ of abelian sheaves for
which $( \prod_{k>0} \mrm{H}^{k}(-,\F^{-k}) )_x \neq 0$ (for instance those described in
[We, A.5] or [MV, 1.30]). Then construct the complex of sheaves $\F$ with zero differential that is 0 in positive degrees and
equal to $\F^{-k}$ in negative degrees. It is not hard to verify that $\rho_F : F \longrightarrow \mbb{H}_{\X} (\F)$ is not a
quasi-isomorphism in this case, so it does not provide a CE-fibrant model for $\F$.

We remark however that
$(\Sh{\X}{\Cochains{R}},\mc{S},\mc{W})$ is still a Cartan-Eilenberg category for any site $\X$: K-injective complexes of sheaves
are easily seen to be CE-fibrant, and by \cite{Sp} each complex of sheaves is locally equivalent to some
$K$-injective one (see also \cite{We}, appendix). Hence the CE-fibrant model of an unbounded
complex $\F$ of sheaves does not agree in general with its hypercohomology sheaf $\mbb{H}_{\X} (\F)$, unless some extra
assumption is imposed on site $\X$.
\end{example}

\subsubsection{} We are going to show that \emph{finite cohomological dimension} is a sufficient condition for the site $\X$ in order that the hypercohomology sheaf $\mbb{H}_{\X}$ produces a resolvent functor for the Cartan-Eilenberg category $(\Sh{\X}{\Cochains{R}},\mc{S},\mc{W})$.

Recall that a \emph{system of neighbourhoods} for a point $x\in \X$ is, by definition, a full cofinal subcategory of
the category of neighbourhoods of $x$ in $\X$ (\cite{SGA4} 6.8.2).

\begin{definition}[\cite{GS}]
A site $\X$ is said to have \textit{finite cohomological dimension} if for any point $x\in\X$ there exists $d\geq 0$ and a system $\Lambda$ of neighbourhoods of $x$ such that for any sheaf of abelian groups $\F\in\Sh{\X}{\Ab}$ and any neighbourhood $U\in\Lambda$ it holds that $\mrm{H}^n(U;\F)=0$ whenever $n > d$.
\end{definition}

For instance, the following sites have finite cohomological dimension:

\begin{enumerate}
\item The small Zariski site of a noetherian topological space of finite Krull dimension; e.g., the Zariski site of a noetherian scheme of finite Krull dimension. This follows from Grothendieck's vanishing Theorem (\cite{Har} III, Theorem 2.7).
\item The big Zariski site of a noetherian scheme $X$ of finite Krull dimension consisting of all schemes of finite type over $X$, or all noetherian schemes of bounded Krull dimension (\cite{GS}, page 6).
\item The small site of a topological manifold of finite dimension. This follows from the vanishing Theorem of \cite{KS}.
\end{enumerate}

\begin{theorem}\label{CEunboundedcomplexes} The descent category $\Cochains{R}$ is
compatible with any  finite cohomological dimension site $\X$.
In this case, properties  \emph{\ref{PropertiesCompat}} hold for $(\Sh{\X}{\Cochains{R}},\Scal,\W)$.
\end{theorem}

The proof is based on a spectral sequence argument, the \emph{Colimit Lemma}, for which we need some preliminaries. The same
spectral sequence argument will also be used in the examples of simplicial sets and spectra.

\subsubsection{}\label{ss} Let $\C$ be a category with filtered colimits and $I$ a filtered indexing set.
For us \lq\lq spectral sequence" means a functorial right half-plane cohomological spectral sequence $E_*$ of abelian groups, commuting with filtered colimits: $E_* (\dirlim_i X_i ) = \dirlim_i E_*(X_i) $.

For an object $X \in \C$, we say that the spectral sequence $E_*(X)$ is {\it bounded on the right\/} if there exists $d$ such that $E_2^{p*}(X) = 0$ for $p>d$. Note that, for conditionally convergent spectral sequences, this implies strong convergence (\cite{Boa}, Theorem 7.4). Given a filtered system $\left\{ X_i\right\}_{i\in I}$ of objects of $\C$, we say that the family of spectral sequences
$\left\{ E_*(X) \right\}_{i\in I}$ is {\it uniformly\/} bounded on the right if there is a fixed $d$ that works for all $i\in I$.

\begin{proposition}[Colimit Lemma]\label{colimitlemma} Assume as given the following data:

\begin{enumerate}
\item[\emph{(1)}] An object $X\in \C$ and a filtered system $X_\bullet = \left\{ X_i\right\}_{i\in I}$ of objects $ X_i \in \C$.
\item[\emph{(2)}] A cone $\left\{ f_i : X_i \longrightarrow X \right\}_{i\in I}$ from the base $X_\bullet$ to the vertex $X$ and, hence, an induced map $f: \dirlim_i X_i \longrightarrow X$.
\end{enumerate}

Moreover, assume also that:

\begin{enumerate}
\item[\emph{(1)}] The spectral sequences $\left\{ E_*X_i\right\}_{i\in I}$ and $E_*X$ converge conditionally to $\left\{ H_i \right\}_{i\in I}$ and $H$, respectively.
\item[\emph{(2)}] The spectral sequences $\left\{ E_*X_i\right\}_{i\in I}$ are uniformly bounded on the right.
\item[\emph{(3)}] The map $E_r(f) : \dirlim_i E_r(X_i) \longrightarrow E_r(X)$ is an isomorphism for some $r \geq 0$.
\end{enumerate}

Then the map $H(f) : \dirlim_i H_i \longrightarrow H$ is an isomorphism too.
\end{proposition}

\begin{proof} See \cite{Mit}, Proposition 3.3.
\end{proof}

\begin{proof}[{\bf Proof of Theorem \ref{CEunboundedcomplexes}}]

The first filtration  of a double complex $K \in \dCochains{R}$,\newline $F^p (\Tot^{\prod} K)^n =
\prod_{s\geq p} K^{s, n-s}$ gives us a
conditionally convergent spectral sequence
$$
E_2^{pq}(\Tot^{\prod} K) = H^p_hH^q_v (K) \quad \Longrightarrow \quad H^{p+q}(\Tot^{\prod} K) \ , \quad p\geq 0 \ .
$$
By Theorem \ref{CartanEilenbergSheaves}, to prove that
for any sheaf $\F \in \Sh{\X}{\Cochains{R}}$ it holds that $\rho_{\F}:\F\longrightarrow \mbb{H}_{\X} (\F)$ is a CE-fibrant model
we may equivalently show that the canonical morphism
$$
\theta_{\F}(x) : \dirlim_{(U,u)\in\mbf{Nbh}(x)} \Tot^{\prod} (G^*\F)(U) \longrightarrow \Tot^{\prod} \dirlim_{(U,u)\in\mbf{Nbh}(x)} (G^*\F)
$$
is a \quis\ of $\Cochains{R}$ for any sheaf $\F$ and any point $x$ in the set of enough points $X$.
These colimits may be computed using the neighbourhoods $(U,u)$ in the
system of neighbourhoods $\Lambda$ that exists by assumption.

Therefore, we have an object
$\Tot^{\prod} x^{\ast}(G^*\F) \in \Cochains{R}$, a filtered system
$\left\{ \Tot^{\prod} (G^*\F)(U)\right\}_{(U,u)}$, where $(U,u)$ runs over all
neighbourhoods of $x$ in $\Lambda$ and the induced map $\theta_{\F}(x)$.

Let us verify the hypotheses of the Colimit Lemma: the spectral sequences
$$
E_2^{pq}(U) = H^p_vH^q_h ((G^*\F)(U)) \quad \Longrightarrow \quad H^{p+q} (\Tot^{\prod} (G^*\F)(U)) \ , \quad p\geq 0
$$
and
$$
E_2^{pq}(x) = H^p_vH^q_h ((G^*\F)_x) \quad \Longrightarrow \quad H^{p+q} (\Tot^{\prod} ((G^*\F)_x)) \ , \quad p\geq 0
$$
converge conditionally.

To compute $E_2^{pq} (U)$ we use that $T:\Sh{\X}{\Cochains{R}}\longrightarrow \Sh{\X}{\Cochains{R}}$ commutes with cohomology in
$\Sh{\X}{\Cochains{R}}$.
At the presheaf level, clearly $H^{\ast} (T (F)) = T(H^{\ast}(F))$ for any presheaf $F$, because cohomology in
$\Cochains{R}$ commutes with products and filtered colimits. Since the stalks of a presheaf $G$ are isomorphic to
the ones of its associated sheaf $G^a$, then $T(G)=T(G^a)$. Hence, if $\F$ is a sheaf
$$T(\Hcal^\ast \F) = T((H^{\ast}\F)^a)= T(H^{\ast}\F)= H^{\ast} (T \F) \ .$$
In particular $H^{\ast} (T \F)= T(\Hcal^\ast \F)$ is a sheaf, so it agrees with its associated sheaf.
Therefore
$\Hcal^{\ast} (T \F)= H^{\ast} (T \F)= T(\Hcal^\ast \F)$, and  $\Hcal^\ast( G^\bullet \F) = G^\bullet (\Hcal^{\ast}\F)$.
We then have, for all $p > d$,
$$
E_2^{pq} (U) = H^p_vH^q_h (G^*\F)(U)
             = H^p (\Gamma (U, \Hcal^q G^*\F) )
             = H^p (\Gamma (U, G^*\Hcal^q\F))
             = H^p (U, \Hcal^q\F)
             = 0
$$
because of the finite cohomological dimension assumption. Finally, already for $r=0$, we have an isomorphism
$$
\dirlim E_0^{pq}(U)  = \dirlim G^p\F^q (U)
                     = (G^p\F^q)_x
                     = E_0^{pq}(x) \ .
$$
Hence the Colimit Lemma tells us that
$$
H^n (\Tot^{\prod} (G^*\F))_x  \longrightarrow H^n \Tot^{\prod} ((G^*\F)_x)
$$
is an isomorphism for all $n$.
\end{proof}

\subsection{Sheaves of fibrant simplicial sets}

\subsubsection{} Let $\D =\Sset_f $ with the descent structure of \ref{example1simplicial}. As in the case of unbounded complexes, the simple functor may not commute weakly with stalks.
Again, for this to hold we must either restrict to simplicial sets with vanishing higher homotopy groups, or
impose some finiteness assumption on the site $\X$. Here we study the second alternative, showing that
$\rho_{\F}:\F\longrightarrow \mbb{H}_{\X} (\F)$ is a CE-fibrant model for each
$\F$ in $(\Sh{\X}{\Sset_f},\mc{W},\mc{S})$ if and only if $\X$ is a site of finite type in the sense of \cite{MV}.

\subsubsection{} By a theorem of Joyal, the category $\Sh{\X}{\Sset}$ possesses a simplicial model category structure in which all objects are cofibrant and the weak
equivalences are the local equivalences \cite{Ja}. The fibrant objects in this
model structure are then defined through a lifting property, and they are objectwise fibrant simplicial sets.
Therefore, there is a fibrant replacement functor $Ex$ that takes a simplicial sheaf to a fibrant one, in
particular $Ex(\F)\in\Sh{\X}{\Sset_f}$.

Given a simplicial sheaf $\F\in \Sh{\X}{\Sset}$ and $n\geq 0$, let $\widetilde{P}^{(n)} \F$ be the simplicial sheaf associated to the presheaf
$U\mapsto P^{(n)} \F(U) = \mrm{Im}\{\F(U)\longrightarrow \mrm{cosk}_n \F(U)\}$. It is equipped with natural maps
$$\F\longrightarrow \widetilde{P}^{(n)} \F\qquad \mbox{ and }\qquad \ \widetilde{P}^{(n+1)} \F \longrightarrow \widetilde{P}^{(n)} \F \ .$$
If the stalks of $\F$ are fibrant simplicial sets, the tower $\{x^{\ast}\widetilde{P}^{(n)}\F={P}^{(n)}x^{\ast} \F\}$
is precisely the Moore-Postnikov tower of $x^{\ast}\F$.
In this case the natural map
$x^{\ast}\F \simeq \invlim_{n\geq 0}\,x^{\ast} \widetilde{P}^{(n)}\F \longrightarrow \holim_{n\geq 0} \,x^{\ast} \widetilde{P}^{(n)}\F$
is a weak equivalence.

\begin{definition}\cite{MV} A site $\X$ is of \emph{finite type} if for each simplicial sheaf $\F\in \Sh{\X}{\Sset}$, the natural morphism
$$ \F \longrightarrow \holim_{n\geq 0} 	\, Ex (\widetilde{P}^{(n)} \F) $$
is a local equivalence.
\end{definition}

\begin{theorem} The descent category $\Sset_f$ is compatible with site $\X$ if and only if $\X$ is of finite type. In this case, properties  \emph{\ref{PropertiesCompat}} hold for $(\Sh{\X}{\Sset_f},\Scal,\W)$.
\end{theorem}

That is, for the site $\X$, either both of the natural morphisms
$$
\holim_{n\geq 0} Ex (\widetilde{P}^{(n)} \F) \longleftarrow \F \longrightarrow \holim_{p\geq 0} G^p\F
$$
are local equivalences for all sheaves $\F$ simultaneously, or neither one is.

\begin{proof} Assume that $\X$ is a site of finite type. If $\F\in \Sh{\X}{\Sset_f}$, since
fibrant objects are preserved by filtered colimits, it holds that $\F$ is locally fibrant.
Then, \cite[1.65]{MV} ensures that $\rho_{\F}:\F\longrightarrow \mbb{H}_{\X} (\F)$ is a local equivalence,
so $\rho_{\F}:\F\longrightarrow \mbb{H}_{\X} (\F)$ is a CE-fibrant model of $\F$.

Conversely, assume that $\rho_{\F}:\F\longrightarrow \mbb{H}_{\X} (\F)$ is a local equivalence for each
$\F\in \Sh{\X}{\Sset_f}$. We must prove that $\F \longrightarrow \holim_{n\geq 0}\, Ex (\widetilde{P}^{(n)} \F)$
is a weak equivalence for each simplicial sheaf $\F$. As in the proof of \cite[1.37]{MV}, taking a suitable
replacement of the tower $\{\widetilde{P}^{(n)} \F\}$ we can assume that $\F$ is a fibrant sheaf, and in particular
that $\F\in \Sh{\X}{\Sset_f}$.

Secondly, to compute $\holim_{n\geq 0}\, Ex (\widetilde{P}^{(n)} \F)$ we may take any choice of fibrant
replacement $Ex$. It follows from \cite[1.59]{MV} that $\mbb{H}_{\X} (\widetilde{P}^{(n)} \F )$ is a fibrant simplicial sheaf for all $n$, so
we may equivalently show that $\F \longrightarrow \holim_{n\geq 0}\, \mbb{H}_{\X} (\widetilde{P}^{(n)} \F)$
is a weak equivalence.

Again, if $\F$ is a simplicial presheaf then $T(\F)$, defined in the same way, is a sheaf
and agrees with $T(\F^a)$, where $\F^a$ is the sheaf associated to the presheaf $\F$. Then
$\mbb{H}_{\X} (\F)$ is isomorphic to $\mbb{H}_{\X} (\F^a)$. In particular,
$\mbb{H}_{\X} (\widetilde{P}^{(n)} \F) = \mbb{H}_{\X} ({P}^{(n)} \F ) $.

Recall that $P^{(n)} \F(U) = (\mrm{Im}\phi_n)(U)$ where $\phi_n:\F\longrightarrow \mrm{cosk}_n \F$, and
that $\mrm{cosk}_n$ is given degreewise by a finite limit. Then, since the Godement resolution and
the simple functor commute with finite limits and images, we have
$\mbb{H}_{\X} (\widetilde{P}^{(n)} \F) = \mbb{H}_{\X} ({P}^{(n)} \F ) =  {P}^{(n)}\mbb{H}_{\X} (\F) $.
In particular, ${P}^{(n)}\mbb{H}_{\X} (\F)$ is already a sheaf, so
$$
\mbb{H}_{\X} (\widetilde{P}^{(n)} \F) = {P}^{(n)}\mbb{H}_{\X} (\F) = \widetilde{P}^{(n)}\mbb{H}_{\X} (\F) \ .
$$
As $\F(U)$ is a fibrant simplicial set, so is $\mbb{H}_{\X} (\F)(U)$. Consequently
$\{{P}^{(n)}\mbb{H}_{\X} (\F) (U) \}_{n}$ is the Moore-Postnikov tower of $\mbb{H}_{\X} (\F)(U)$ and
$\mbb{H}_{\X} (\F)(U)\longrightarrow \holim_{n\geq 0}\, {P}^{(n)}\mbb{H}_{\X} (\F) (U) $ is a weak equivalence of simplicial sets.
But then
$$\mbb{H}_{\X} (\F)\longrightarrow \holim_{n\geq 0}\, \widetilde{P}^{(n)}\mbb{H}_{\X} (\F)=\holim_{n\geq 0}\, \mbb{H}_{\X} (\widetilde{P}^{(n)} \F)
$$
is a global equivalence. Composing with the local equivalence $\F\longrightarrow \mbb{H}_{\X} (\F)$ we
conclude that $\F\longrightarrow \holim_{n\geq 0}\, \mbb{H}_{\X} (\widetilde{P}^{(n)} \F)$ is a local equivalence as required.
\end{proof}

Finally, let us remark that finite cohomological dimension implies finite type:

\begin{proposition} Every finite cohomological dimension site is of finite type.
\end{proposition}

\begin{proof}
This may be seen as in \cite{MV}, Theorem 1.37, or using an spectral sequence argument such as the one given below for spectra, and previously for unbounded complexes.
\end{proof}

\subsection{Sheaves of fibrant spectra}

\subsubsection{}
Let $\D = \Sp_f$ with the descent structure of \ref{example1simplicial}. Then Theorem \ref{CartanEilenbergSheaves} applies and gives a CE structure to $\Sh{\X}{\Sp_f}$, if $\X$ is of finite cohomological dimension.

\begin{theorem} The descent category $\Sp_f$ is
compatible with any finite cohomological dimension site $\X$.
In this case, properties  \emph{\ref{PropertiesCompat}} hold for $(\Sh{\X}{\Sp_f},\Scal,\W)$.
\end{theorem}

\begin{proof}
Let $X^\bullet \in \Delta\Sp_f$ be a cosimplicial object of $\Sp_f$. According to \cite{Bous}, 2.9 (see also \cite{Hir}, remark 18.1.11), we have a conditionally convergent spectral sequence (\cite{Boa})
$$
E_2^{pq} (X^\bullet) = \pi^p\pi_q (X^\bullet) \quad \Longrightarrow\quad \pi_{q-p} (\holim X^\bullet ) \ , \quad p\geq q\geq 0 \ .
$$
As in the case of unbounded complexes, we want to show that the canonical morphism
$$
\theta_\F (x) : \dirlim_{U \in \mathbf{Nbh}(x) } \holim_p (G^p \F)(U) \longrightarrow \holim_p \dirlim_{U \in \mathbf{Nbh}(x) } (G^p \F)(U)
$$
is a weak equivalence of $\Sp_f$ for any sheaf $\F$ and any point $x\in X$ and the colimit may be computed using the neighbourhoods $(U,u)$ in the system $\Lambda$
which exists by hypothesis. Let us verify the hypotheses of the Colimit Lemma: the spectral sequences
$$
E_2^{pq}(U) = \pi^p\pi_q ((G^\bullet \F)(U)) \quad \Longrightarrow \quad \pi_{q-p} (\holim (G^\bullet \F)(U)) \ , \quad p\geq q\geq 0
$$
and
$$
E_2^{pq}(x) = \pi^p\pi_q ((G^\bullet \F)_x) \quad \Longrightarrow \quad \pi_{q-p} (\holim (G^\bullet \F)_x) \ , \quad p\geq q\geq 0
$$
converge conditionally. Moreover, the Godement cosimplicial resolution commutes with homotopy groups
(argue as in the case of cohomology, or see \cite{Th}, page 452, formula (1.26)). Hence, for all $p > d$,
$$
E_2^{pq}(U) = \pi^p\pi_q ((G^\bullet \F)(U))
            = H^p (\Gamma (U, \pi_q (G^\bullet \F)))
            = H^p (\Gamma (U, G^\bullet \pi_q (\F)))
            = H^p (U, \pi_q (\F))
            = 0  \ .
$$
Finally, for $r=2$, we have an isomorphism
\begin{align*}
\dirlim E_2^{pq}(U)    &=    \dirlim \pi^p\pi_q ((G^\bullet \F)(U))
                       =    H^p (\dirlim \Gamma (U, \pi_q (G^\bullet \F)))  \\
                          &=    \pi^p (\pi_q (G^\bullet \F)_x)
                          =    \pi^p \pi_q  ((G^\bullet \F)_x )
                          =    E_2^{pq}(x) \ .
\end{align*}
Hence, the Colimit Lemma tells us that
$$
\pi_{q-p} ((\holim G^\bullet \F)_x)  \longrightarrow  \pi_{q-p} (\holim (G^\bullet \F)_x )
$$
is an isomorphism for all $p\geq q \geq 0$.
\end{proof}

\subsection{Sheaves of filtered complexes}

\subsubsection{} As happens in the case of bounded complexes, the simple functors described in Example \ref{examplefilteredcomplexes} clearly commutes with the formation of stalks. Hence

\begin{theorem}\label{CEfilteredcomplexes} Assume that $\A$  is an abelian category satisfying $($AB4$)^{\ast}$ and $($AB5$)$.
Then, the descent categories $(\mrm{F}\Cochainsp{\A},\mrm{E}_r)$ are compatible with any site $\X$.
In particular, properties  \emph{\ref{PropertiesCompat}} hold for $(\Sh{\X}{\mrm{F}\Cochainsp{\A}},\Scal_r,\W_r)$.
\end{theorem}

Note that $\Sh{\X}{\mrm{F}\Cochainsp{\A}}$ agrees with the category of filtered complexes of sheaves on $\A$.
For $r=0$, $\Sh{\X}{\mrm{F}\Cochainsp{\A}}[\mrm{E}_0^{-1}]$ is then $\mrm{F}\mc{D}^{\geq b}(\Sh{\X}{\A})$, the filtered derived category of sheaves on $\A$. That is, for $\A = R$-modules, our sheaf cohomology agrees with classic filtered sheaf cohomology.

\section{Varying $\X$ and $\D$}

In this section we are going to show two elementary results regarding the behaviour of our derived functors
$\mathbb{R}f_*$ and $\mathbb{R}\Gamma$ under the action of morphisms of sites $f: \X \longrightarrow \Y$ and coefficients
$\phi: \D \longrightarrow \D'$.

\subsection{Varying $\X$}

\subsubsection{} Throughout the rest of this section, $(\mc{D}, \mathrm{E})$ is a descent category. All the sites appearing here are
assumed to be compatible with it. We are going to prove in our general setting the following result about the variance of sheaf
cohomology under a morphism of sites.

\begin{theorem}\label{variandoX}Let $(\D, \mathrm{E})$ be a descent category and $\X$,  $\Y$ compatible sites. Then, for any continuous morphism $g:\Y\longrightarrow \Z$ and any geometric morphism $f:\X\longrightarrow \Y$ of sites, we have natural isomorphisms of functors
$$
\mbb{R} (g\circ f)_* =  \mbb{R}g_* \circ \mbb{R}f_* = g_* \circ \mbb{R}f_* \ .
$$
Particularly, for any object $V \in \Y$
$$
\mathbb{R}\Gamma (f^{-1}V,\ - )  = \mathbb{R}\Gamma (V,\ - ) \circ \mathbb{R}f_* = \Gamma (V ,\ - ) \circ \mathbb{R}f_* \ .
$$
\end{theorem}

\begin{corollary} With the same hypotheses, we also have natural isomorphisms of functors
$$
\mathbb{R}\Gamma (\X, \F) = \mathbb{R}\Gamma(\Y, \,\mbb{R}f_* \F) = \Gamma(\Y, \,\mbb{R}f_* \F)    \ .
$$
\end{corollary}

The proof requires some technical results that we are going to state and prove first in the next sections.

\subsubsection{} To begin with, notice that, in abelian sheaf cohomology the theorem would typically be proved resorting to a Leray spectral sequence-type argument. Since we are not necessarilly working with abelian sheaves, spectral sequences are not available to us in our general situation, but Grothendieck's \lq\lq chain rule" for derived functors plays an equivalent role and, in fact, proves to be easier to handle.

\begin{lemma}[Chain Rule]\label{chainrule} Let
$$
\A \stackrel{F}{\longrightarrow} \B \stackrel{G}{\longrightarrow} \C
$$
be a pair of functors between CE-categories such that:

\begin{enumerate}
\item[(i)] $F$ preserves strong equivalences and CE-fibrant objects.
\item[(ii)] $G$ sends strong equivalences to weak ones.
\end{enumerate}

Then the right derived functors $\mathbb{R}F, \mathbb{R}G$ and $\mathbb{R}(G\circ F)$ exist and there are natural isomorphisms of functors

$$
\mathbb{R}(G\circ F) = \mathbb{R}G \circ \mathbb{R}F = G \circ \mathbb{R}F \ .
$$
\end{lemma}

\begin{proof} For any object $A\in \A$, we choose a fibrant model $A \longrightarrow M\in \W_\A$ and we have
$$
(\mathbb{R}G\circ\mathbb{R}F)(A) = \mathbb{R}G(FM) \ .
$$
Now, in order to compute this last derived functor, we also choose a CE-fibrant model for $FM$: $FM \longrightarrow N \in \W_\B$. But, since $FM$ is already fibrant by hypothesis, the Whitehead theorem \cite[Th. 2.2.5]{GNPR1} tells us that this last weak equivalence is in fact a strong one. Hence, by hypothesis, $GFM \longrightarrow GN \in \W_\C$. So, we conclude

$$
\mathbb{R}G(FM) = GN = GFM = \mathbb{R}(G\circ F)(A) \ .
$$

And since $FM$ is already fibrant, we also have $\mathbb{R}G(FM) = GFM = (G \circ \mathbb{R}F) A $.
\end{proof}

\subsubsection{} The fact that the (realizable) homotopy limit of a degree-wise local object of a homotopical localization is a local object as well translates to CE-categories as

\begin{lemma}\label{simpledefibrantes} Let $\F^\bullet$ be a cosimplicial sheaf such that $\F^p$ is CE-fibrant
for all $p\geq 0$. Then $\mbf{s}(\F^\bullet)$ is also a CE-fibrant sheaf.
\end{lemma}

\begin{proof}
The resolvent functor $(\mbb{H}_{\X},\rho)$ for $(\Sh{\X}{\D},\mc{S},\mc{W})$ induces a degree-wise functor on $(\Dl\Sh{\X}{\D}$ $=$ $\Sh{\X}{\Dl\D},\mc{S},\mc{W})$, that is  a CE-category in which a cosimplicial sheaf $\F^{\bullet}$ is fibrant if and only if $\rho_{\F^{\bullet}}$ is
degreewise in $\mc{S}$. In particular a cosimplicial sheaf which is degreewise fibrant is also fibrant in $\Sh{\X}{\Dl\D}$.
Now, $\simple\F^{\bullet}$ is fibrant if and only if given a local equivalence $w:\G\longrightarrow \G'$ the induced map
$$w^*: \loc{\Sh{\X}{\D}}{S} (\G',\simple \F^\bullet) \longrightarrow \loc{\Sh{\X}{\D}}{S} (\G,\simple \F^\bullet)$$
is a bijection. But by adjunction the above map is isomorphic to
$$w^*: \loc{\Sh{\X}{\Dl\D}}{S} (c\G', \F^\bullet) \longrightarrow \loc{\Sh{\X}{\Dl\D}}{S} (c\G,\F^\bullet)\ , $$
which is a bijection because $\F^\bullet$ is fibrant.
\end{proof}

\subsubsection{} Now we prove the following result generalizing the well-known fact in the abelian context that $f_*$ preserves injective sheaves.

\begin{lemma}\label{preservaimagenesdirectas} $f_*$ preserves CE-fibrant sheaves.
\end{lemma}

\begin{proof} A geometric morphism $f:\X \longrightarrow \Y$ sends points to points:
if $x = (x^*,x_*)$ is a point of $\X$, then  $(x^* f^*, f_* x_*)$ is a point of $\Y$, to be denoted by $f(x)$.
We assume that there are sets of enough points $X,Y$ in our sites $\X, \Y$ such that $f(X) \subset Y$.

Given a sheaf $\F \in \Sheaf (\X, \D)$, we must show that
$f_*\mathbb{H}_\X\F \in \Sheaf (\Y,\D)$ is fibrant. Since direct images commute with simple functors by definition,

$$
f_* \mbb{H}_{\X} \F = f_* \simple G^\bullet \F= \simple f_* G^\bullet \F \ ,
$$

because of Lemma \ref{simpledefibrantes}, we only need to prove that sheaves like $f_*T_X\F = f_*p_*^Xp^*_X\F$ are CE-fibrant. To this end, we will show that sheaves of the form $f_*p_*^XD$, for any $D \in \D^X$, are fibrant. Being $f_*$ a right adjoint,
it commutes with limits. Then

$$
f_*p_*^XD = f_* \left( \prod_{x\in X}x_*(D_x) \right) = \prod_{x\in X} f(x)_*(D_x) \ .
$$

Since the product of CE-fibrant sheaves is a CE-fibrant sheaf, it suffices to prove that each $f(x)_*(D_x)$ is fibrant. But given any point $y\in\Y$ and any object $D_0\in\mc{D}$, it is indeed true that $y_* D_0$ is fibrant. To see this, consider the object $D\in\mc{D}^Y$ which maps $y$ to $D_0$ and any other point in $Y$ to
the final object $*$ of $\mc{D}$. Therefore, by definition
$$
p_*^YD = \prod_{z\in Y} z_* D_z = y_* D_0\ ,
$$
and the result follows then from the following lemma.
\end{proof}

\begin{lemma}\label{losflacidossonCEfibrantes} For any $D\in \D^Y$, $p^Y_*D \in \Sheaf (\Y,\D)$ is a CE-fibrant sheaf.
\end{lemma}

\begin{proof}
Lemma \ref{DegeneracionExtraGodement} tells us that $\eta_{p^Y_*D} : p^Y_*D \longrightarrow G^\bullet p^Y_*D$ has an
extra degeneracy and hence is a cosimplicial homotopy equivalence. Therefore,
$ p^Y_*D \longrightarrow \simple G^\bullet p^Y_*D = \mathbb{H}_\Y p^Y_*D$ is a global one. That is, the sheaf $p^Y_*D$ satisfies Thomason's descent. Therefore, Corollary \ref{Thomasondescentproperty} tells us it is CE-fibrant.
\end{proof}

\subsubsection{} With everything in place, we can now prove our theorem.

\begin{proof}[Proof of Theorem \ref{variandoX}] The first statement follows directly from
Lemmas \ref{chainrule}  and \ref{preservaimagenesdirectas}. To see the second one, let $V \in \Y$ be any object. Then

\begin{eqnarray*}
\mbb{R}\Gamma (f^{-1}V, \F) &=& \Gamma(f^{-1}V, \mbb{H}_{\X}\F) \qquad \text{by definition of}\ \mathbb{R}\Gamma (f^{-1}V,\F) \\
    &=& \Gamma(V, f_*\mbb{H}_{\X}\F) \qquad \text{by definition of}\ f_*  \\
    &=& \Gamma(V, \mbb{R}f_* \F) \qquad \text{by definition of}\ \mathbb{R}f_* \\
    &=& \mathbb{R}\Gamma (V, \mathbb{R}f_*\F) \qquad \text{because of the chain rule \ref{chainrule}} \ .
\end{eqnarray*}
\end{proof}

\subsection{Varying $\D$}

\subsubsection{} Next we study the preservation of the CE-structures previously obtained on sheaves under a change of category of coefficients $\D$.

That is, we consider a fixed Grothendieck site $\X$ and two categories of coefficients $\D, \D'$. A functor $\phi : \D \longrightarrow \D'$ naturally induces a functor $\phi : \PrSh{\X}{\D}\longrightarrow \PrSh{\X}{\D'}$ between the categories of presheaves. If $\phi$ preserves limits, then it restricts to a functor $\phi : \Sheaf (\X,\D) \longrightarrow \Sheaf (\X,\D')$ between the categories of sheaves.

Now, let us assume that both pairs $(\X, \D)$ and $(\X,\D')$ are compatible, so both categories of sheaves have CE-structures. The next result is a situation where such a $\phi$ preserves these CE-structures.

\begin{theorem} Assume $(\D,\mrm{E})$ and $(\D',\mrm{E}')$ are compatible with site  $\X$ and let
$\phi:\mc{D}\longrightarrow\mc{D'}$ be a functor such that
\begin{enumerate}
 \item[\emph{(1)}] $\phi$ preserves limits.
 \item[\emph{(2)}] If $I$ is a filtered small category, the canonical morphism $\dirlim_I \phi\longrightarrow \phi\dirlim_I$ is in $\mrm{E}'$.
 \item[\emph{(3)}] $\phi(\mrm{E})\subset \mrm{E}'$.
 \item[\emph{(4)}] There is a zigzag of natural weak equivalences connecting $\phi\mbf{s}$ and $\mbf{s}\phi$.
 \end{enumerate}
Then $\phi:\Sh{\X}{\D}\longrightarrow \Sh{\X}{\D'}$ is a morphism of CE-categories. That is, it preserves global and local equivalences and sends fibrant sheaves to
fibrant sheaves.
In particular, if $\Y$ is another site compatible with both descent categories, then, for any continuous morphism $f:\X\longrightarrow \Y$ of sites,
$$ \mbb{R}f_* \phi =  \phi \,\mbb{R}f_*  \,\,\, .$$
\end{theorem}

\begin{proof} (1) ensures that the induced $\phi$ at the level of presheaves restricts to sheaves. We have that $\phi(\mc{S})\subset \mc{S}'$ by (3) , and $\phi(\mc{W})\subset \mc{W}'$ by (2) and (3). Now, (1) implies that
$\phi$ commutes with $p_*$ and (2) that it commutes up to natural weak equivalence with $p^*$, so we have an induced canonical weak equivalence $G^\bullet \phi \longrightarrow \phi G^\bullet$ which in conjunction with (4) yields a canonical isomorphism $\mbb{H}_{\X}\,\phi\longrightarrow \phi\,\mbb{H}_{\X}$ on $\Sh{\X}{\D'}[\mc{S}'^{-1}]$. Consequently, given a sheaf $\F\in \Sh{\X}{\D}$, we conclude that $\phi(\rho_{\F}) : \phi(\F)\longrightarrow \phi\mbb{H}_{\X}(\F)$ is in $\mc{S}'$ if and only if
$\rho_{\phi(\F)}:\phi(\F)\longrightarrow \mbb{H}_{\X}\phi(\F)$ is.

If $\F$ is a fibrant sheaf this means, by Corollary \ref{Thomasondescentproperty}, that $\rho_{\F}$ is in $\mc{S}$, so  $\phi(\rho_{\F})$ is
in $\mc{S}'$. But then $\rho_{\phi(\F)}$ belongs to $\mc{S}'$ as well, which means by Corollary \ref{Thomasondescentproperty} that $\phi(\F)$
is fibrant.
\end{proof}

\begin{corollary}
Under the same hypotheses of the previous proposition, for any object $U\in \X$ and any sheaf $\F \in \Sheaf (\X,\D)$, we have
$$
\mathbb{R}\Gamma (U,\phi\F ) = \phi\mathbb{R}\Gamma (U,\F) \ .
$$
\end{corollary}

\begin{example} As we have seen in the previous section of examples, the descent categories $(\Cochainsp{\A},\mrm{E})$ and $(\mrm{F}\Cochainsp{\A},\mrm{E}_r)$ of (filtered) complexes of Examples \ref{example1bounded}, \ref{examplefilteredcomplexes} are
compatible with any site, provided $\A$ is $($AB4$)^{\ast}$ and $($AB5$)$.
Then the above result applies to both the forgetful functor $U:(\mrm{F}\Cochainsp{\A},\mrm{E}_0)\longrightarrow (\Cochainsp{\A},\mrm{E})$ and the
decalage filtration functor $Dec:(\textbf{F}\Cochainsp{\A},\mrm{E}_{r+1})\rightarrow(\textbf{F}\Cochainsp{\A},\mrm{E}_{r})$. This in particular recovers
the classic result that filtered sheaf hypercohomology and filtered higher direct images agree with the usual abelian ones when we forget the filtrations.
\end{example}

Consequently, we can now extend \ref{transferLema} to categories of sheaves, obtaining a transfer lemma for CE-structures between them.

\begin{proposition}  Assume that $\D$ is closed under products and filtered colimits, and that $(\D',\mrm{E}')$ is a descent category
compatible with the site  $\X$.
If $\psi:\mc{D}\longrightarrow\mc{D'}$ satisfies the hypotheses of the transfer lemma \ref{transferLema} and \emph{(1)}
and \emph{(2)} of the previous proposition, then
\begin{enumerate}
 \item[\emph{(I)}] $(\D,\mrm{E}=\psi^{-1}\mrm{E}')$  is a descent category compatible with $\X$.
 \item[\emph{(II)}] $\psi:\Sh{\X}{\D}\longrightarrow \Sh{\X}{\D'}$ is a morphism of CE-categories.
\end{enumerate}
 \end{proposition}

\begin{proof} By the transfer lemma \ref{transferLema} and the previous proposition, the only statement remaining to be proved is the fact that
the resulting descent category $(\D,\mrm{E}=\psi^{-1}\mrm{E}')$  is compatible with $\X$. First note that by
definition $(\D,\mrm{E})$ satisfies hypotheses (\ref{hipotesis}). Then, using Theorem \ref{CartanEilenbergSheaves}, it suffices to
show that for any $\F\in \Sh{\X}{\D}$, $\rho_{\mbb{H}_{\X}\F}$ is in $\mc{S}$. By definition of $\mrm{E}$, this holds if and only if
$\psi(\rho_{\mbb{H}_{\X}\F})$ is in $\mc{S}'$. But arguing as in the previous proof, this happens if and only if
$\rho_{\mbb{H}_{\X}\psi(\F)}$ is in $\mc{S}'$, which holds because $(\D',\mrm{E}')$ is compatible with $\X$.
\end{proof}

\begin{examples} In fact, the functors in the previous example also satisfy these stronger hypotheses and then may be used to transfer compatibility with the site.

To close the paper, let us briefly describe a classical situation also covered by this transfer lemma. It's a well-known fact that, if $(X,\Ocal_X)$ is a ringed space, then the derived functor $\mathbb{R}\Gamma (X,\F) $ naturally inherits a module structure for any sheaf $\F$ of $\Ocal_X$-modules. But, if we forget this module structure through the forgetful functor $\psi : \mathbf{Mod} \longrightarrow \Ab$, this derived functor agrees with the usual cohomology as an abelian sheaf: $\mathbb{R}\Gamma(X,\psi\F) = \psi\mathbb{R}\Gamma(X,\F)$. In a forthcoming article, we will show that an analogous result holds for sheaves of operad algebras.
\end{examples}

\end{large}
\end{document}